\pdfoutput=1

\def\Cent          {\mathrm{Cent}}
\def\kk            {{\ensuremath\mathbbm{k}}}
\def\C             {{\ensuremath{\mathbb C}}}
\def\R             {{\ensuremath{\mathbb R}}}
\def\Q             {{\ensuremath{\mathbb Q}}}
\def\Z             {{\ensuremath{\mathbb Z}}}
\def\N             {{\ensuremath{\mathbb N}}}

\def\g             {{\ensuremath{\mathfrak g}}}
\def\Sl            {{\ensuremath\mathfrak{sl}}}
\newcommand\car[1] {\langle#1\rangle}
\def\df            {\,{:=}\,}

\def\eq            {\,{=}\,}
\def\Hom           {{\ensuremath{\mathrm{Hom}}}}

\def\Id            {\mbox{\sl Id}}
\def\id            {\mbox{\sl id}}
\def\one           {{\bf1}}

\documentclass[a4paper, 11pt, reqno]{amsart}
\usepackage{amsmath, amsthm, amsfonts, amssymb} 
\usepackage{stmaryrd}
\usepackage{bbm}
\usepackage{enumerate}
\usepackage{graphicx,color,pifont} 
\usepackage{pifont,color,tikz,tikz-cd}
\usepackage{multirow,longtable,float}
\usepackage{pdfpages} 
\usepackage{hyperref}
\usepackage[left=25mm,right=35mm, top=30mm, bottom=35mm,
includeheadfoot]{geometry}

\newtheorem{thm}{Theorem}
\newtheorem{conv}[thm]{Convention}
\newtheorem{ass}[thm]{Assumption}

\newtheorem{lem}[thm]{Lemma}

\theoremstyle{definition}
\newtheorem{expl}[thm]{Example}
\newtheorem*{Expl}{Example}

\newtheorem{defi}[thm]{Definition}

\newtheorem{rem}[thm]{Remark}
\newtheorem{Rem}[thm]{Remark}
\newtheorem{que}[thm]{Question}

\newenvironment{narrow}[2]{
\begin{list}{}{%
 \setlength{\topsep}{0pt}%
 \setlength{\leftmargin}{#1}%
 \setlength{\rightmargin}{#2}%
 \setlength{\listparindent}{\parindent}%
 \setlength{\itemindent}{\parindent}%
 \setlength{\parsep}{\parskip}%
}%
\item[]}{\end{list}}

\newcommand{\hamburger}[4] 
{
  \thispagestyle{empty}
  \vspace*{-2cm}
  \begin{flushright}
    ZMP-HH / #2 \\
    Hamburger Beitr{\"a}ge zur Mathematik Nr. #3 \\
    #4 \\
  \end{flushright}
  \vspace{0.5cm}
  \begin{center}
    \Large \bf
    #1
  \end{center}
  \vspace{0.5cm}
  \begin{center}	
    Simon Lentner\footnote{Corresponding author:\\
      \texttt{simon.lentner@uni-hamburg.de},
    Phone: +49 40 42838 5178, Fax: +49 40 42838-5190}
    and Daniel Nett \\
    Algebra and Number Theory, 
    Hamburg University,\\
    Bundesstra{\ss}e 55, D-20146 Hamburg \\
  \end{center}
  \vspace{0.2cm}
}

\newcommand{\Multi}[2]{\multirow{#1}{*}{#2}}
\begin{document}

\enlargethispage{5\baselineskip}
\numberwithin{equation}{section}
\numberwithin{thm}{section}
                                     
\hamburger{New $R$-matrices for small quantum groups}%
          {14-19}{524}{September 2014}

\begin{abstract}
  It is widely accepted that small quantum groups should possess a quasitri-
  angular structure, even though this is technically not true. In this article
  we construct explicit $R$-matrices, sometimes several inequivalent ones, over
  certain natural extensions of small quantum groups by grouplike elements.
  The extensions are in correspondence to lattices between root and weight
  lattice. Our result generalizes a well-known calculation for $u_q(\Sl_2)$
  used in logarithmic conformal field theories. \\

\end{abstract}
\title{}
\date{}
\maketitle

\tableofcontents
\newpage

\section*{Introduction and Summary}
  Hopf algebras with $R$-matrices, so called quasitriangular Hopf algebras,
  give rise to braided tensor categories, which have many interesting
  applications: Any braided vector space with a dual can be used to construct
  knot invariants and, using surgery, a (finite) braided tensor category gives
  rise to a invariant of $3$-manifolds, cf. \cite{Vir06} based on the well-known
  work \cite{RT90}. In \cite{Ros93,KR02} the case of the representation
  category of a quantum group is treated. For example, if the $R$-matrix for the quantum 
  group $U_q(\g)$ in the case $q=i,\g=\Sl_2$ is evaluated on the standard representation 
  depending on an additional deformation parameter $\lambda$, then one obtaines in this   
  way the Alexander-Conway-polynomial. Braided tensor categories with an
  additional non-degeneracy condition give rise to topological field theories
  \cite{Tur94, KL01}. Checking which $R$-matrices below fulfill this additional condition 
  would be an interesting follow-up to the present work.

  For quantum groups, Lusztig gives in \cite{Lus93} Sec. 32 essentially an
  $R$-matrix, but it is not clear that this gives rise to
  an $R$-matrix over the small quantum groups $u_q(\g)$ with $q$ an $\ell$-th
  root of unity. In \cite{Ros93} this has been shown to be true whenever
  $\ell$ is odd and prime to the determinant of the Cartan matrix. In other cases 
  Lusztig's small quantum group itself usually does not admit an $R$-matrix, in 
  many cases even the category is not braided. This has been resolved in
  two ways in literature:
 
  \begin{itemize}
    \item Several authors consider slightly smaller quotients (resp. a
      subcategory), i.e. $K^e=1$
      for $e$ half the exponent in Lusztig's definition,  where one can obtain
      indeed an $R$-matrix if $\ell$ is prime to the determinant of the Cartan
      matrix \cite{Ros93}. For some applications however, it is desirable that the
      quotient is taken precisely with Lusztig's choice and one wishes to
      focus on the even case.
    \item For $q$ an even root of unity, some authors consider $R$-matrices
      up to outer automorphism (\cite{Tan92,Res95}), or quadratic
      extensions of $u_q(\g)$, e.g. explicitly in the case of $u_q(\Sl_2)$ in
      \cite{RT91,FGST06} and more generally in \cite{GW98} for $u_q(\Sl_n)$.
      By \cite{Tur94} p. 511 Rosso has already suggested in 1993 that one
      should consider extensions of $u_q(\g)$ for general $\g$. 
  \end{itemize}
  
  In this article we determine \emph{all} possible $R$-matrices that can be obtained 
  through
  Lusztig's ansatz \cite{Lus93} Sec. 32.1, which means to vary the
  \emph{toral part} $R_0$ (see below), while at the same
  time considering extensions of $u_q(\g)$ that are Lie-theoretically
  motivated and explain the exceptional behaviour with respect to the
  determinant of the Cartan matrix.
  In many cases we find several inequivalent choices different from the
  standard choice of $R_0$ (most notably $\g=D_{2n}$), while other cases still
  do not admit $R$-matrices. In particular we find indeed that also even
  $\ell$ (or
  divisible by $4$ for multiply-laced $\g$) admit $R$-matrices
  for extensions of Lusztig's original quantum group.

  More precisely, the extensions $u_q(\g,\Lambda)$ of $u_q(\g)$ we consider
  depend on a choice of a
  lattice $\Lambda_R\subset\Lambda\subset\Lambda_W$ between root and weight
  lattice, which corresponds to a choice of a complex connected Lie group
  associated to $\g$. We first derive a necessary form of the $R$-matrix,
  depending only on the fundamental group $\Lambda_W/\Lambda_R$; this
  amounts to a question in additive combinatorics we have settled in 
  \cite{LN14}. The main calculations concluding the present article is to
  check sufficiency in terms of certain sublattices of $\Lambda$. These
  sublattices depend heavily on $\g$ and on the roots of unity in
  question, in particular in common divisors of $\ell$ and the determinant of
  the Cartan matrix, which is the order of  $\Lambda_W/\Lambda_R$.  

  \noindent
  This article is organized as follows.
 
  In Section \ref{prelim} we fix the Lie theoretic notation and prove some
  technical preliminaries. In particular, we introduce some sublattices of the
  weight lattice $\Lambda_W$ of a simple complex Lie algebra, e.g. the
  so-called $\ell$-centralizer $\Cent^q(\Lambda_R)$ of $\Lambda_R$ in
  $\Lambda_W$ (with respect to the braiding). We
  then give the definition of the finite dimensional quantum groups
  $u_q(\g,\Lambda,\Lambda')$ for lattices $\Lambda,~\Lambda'$, where
  $\Lambda'$ is a suitable sublattice of $\Cent^q(\Lambda_R)$. Choices
  of $\Lambda'$ correspond to the choice of a quotient, see above. We recall
  also the definition of an $R$-matrix.
 
  In Section \ref{ansatz} we review the ansatz $R=R_0\bar\Theta$ for
  $R$-matrices by Lusztig, with fixed $\Theta\in u_q(\g,\Lambda)^+\otimes
  u_q(\g,\Lambda)^-$ and free toral part
  $R_0=\sum_{\mu,\nu\in\Lambda/\Lambda'}f(\mu,\nu)K_\mu\otimes K_\nu$. We find
  equations for the free parameters $f(\mu,\nu)$ that are equivalent to $R$
  being an $R$-matrix and depend on the fundamental group
  $\pi_1=\Lambda_W/\Lambda_R$ of $\g$ and on some sublattices of $\Lambda_W$
  associated to $q$.  This ansatz was also used by M\"uller \cite{Mue98a,
  Mue98b} for determining $R$-matrices for quadratic extensions of
  $u_q(\Sl_n)$.
 
  In Section \ref{group-equations} we will first consider those equations on
  $f(\mu,\nu)$, that only
  depend on $\pi_1$ as a group, the so-called \emph{group-equations} for the
  coefficients of the ansatz in Section \ref{ansatz}.  We will give all
  solutions of the group-equations of a group $G$, where $G$ is cyclic or
  equal to $\Z_2\times\Z_2$, since these are the relevant cases for $G=\pi_1$
  the fundamental group of the Lie algebras in interest. 
  The case $\g=A_n$ with fundamental group $\Z_{n+1}$ is particularly hard
  and depends on a question in additive combinatorics, which we settled
  in \cite{LN14}.

  We then consider in Section \ref{diamond-equations} a certain 
  constellation of sublattices of $\Lambda$, which we call a diamond. Depending on these  
  sublattices we define \emph{diamond-equations}, 
  derive a necessary condition for the existence of solutions and give again 
  results for the cyclic case. 
 
  In Section \ref{solutions} we give the main result of this article in Theorem
  \ref{Thm:solutions}, a list of $R$-matrices obtained by Lusztig's ansatz.
  These are obtained by solving the corresponding group- and
  diamond-equations, depending on the fundamental group $\pi_1$ of $\g$, the
  lattice $\Lambda_R\subset \Lambda\subset \Lambda_W$, kernel
  $\Lambda'\subset\Cent^q(\Lambda_W)\cap\Lambda_R$ and the $\ell$-th root of
  unity $q$. Here, $\Cent^q(\Lambda_W)$ denotes the lattice orthogonal to
  $\Lambda_W$ mod $\ell$, i.e. the set of $\lambda\in\Lambda$ with
  $\ell\mid(\lambda,\mu)$ for all weights $\mu$.

  We develop general results that allow us to compute all $R$-matrices
  fulfilling Lusztig's ansatz depending on $\g,\Lambda,\Lambda'$. Under the additional 
  assumption \ref{ass:Lambda'} on $\Lambda'$, which also simplifies
  some calculations, we find that in fact 
  $\Lambda'=\Lambda_R^{[\ell]}$ is the only choice that allows the existence of an $R$-   
  matrix.
  
  \renewcommand{\thethm}{\Alph{thm}}
  \begin{thm}\label{Thm:solutions}
    Let $\g$ be a finite-dimensional simple complex Lie algebra with root
    lattice $\Lambda_R$, weight lattice $\Lambda_W$  and fundamental group
    $\pi_1=\Lambda_W/\Lambda_R$. Let $q$ be an $\ell$-th root of unity,
    $\ell\in\N$, $\ell>2$. 
    Then we have the following $R$-matrix of the form $R=R_0\bar\Theta$, with
    $\Theta$ as in Theorem \ref{Thm:Theta}:
    \begin{equation*}
      R = 
      \left(
      \frac 1{|\Lambda/\Lambda'|}
      \sum_{(\mu,\nu)\in (\Lambda_1/\Lambda'\times\Lambda_2/\Lambda')}
      q^{-(\mu,\nu)} \omega(\bar\mu,\bar\nu) ~ 
      K_{\mu}\otimes K_{\nu}
      \right)
      \cdot \bar\Theta, 
    \end{equation*}
    for the quantum group $u_q(\g,\Lambda,\Lambda')$ with $\Lambda_i$ the
    preimage of a certain subgroup $H_i\subset\pi_1$ in $\Lambda_W$ ($i=1,2$),
    a certain group-pairing $\omega\colon H_1\times H_2\to \C^{\times}$ and
    $\Lambda'=\Lambda_R^{[\ell]}$ as in Def. \ref{elllattice}. 
    
    In the following table we list for all root systems the following data,
    depending on $\ell$: Possible choices of $H_1,H_2$ (in terms of
    fundamental weights $\lambda_k$), the group-pairing $\omega$, and the
    number of solutions $\#$.  If the number has a superscript $*$, we obtain
    $R$-matrices for Lusztig's original choice of $\Lambda'$. For $\g=D_n,\,
    2\mid n$, with $\pi_1=\Z_2\times\Z_2$ we get the only cases $H_1\neq H_2$
    and denote by
    $\lambda\neq\lambda'\in\{\lambda_{n-1},\lambda_n,\lambda_{n-1}+\lambda_n\}$
    arbitrary elements of order $2$ in $\pi_1$.
    \\

    \enlargethispage{.5cm}

  \renewcommand{\arraystretch}{1.1}
  \begin{longtable}{c||c|c||c|c|l}
    \multicolumn{1}{c||}{$\g$} &$\ell$ &\# &$H_i\cong$
    &$H_i\,{\scriptstyle (i=1,2)}$ 
                               &\multicolumn{1}{c}{$\omega$}
      \\ \hline \hline
      
      & \Multi{2}{$\ell$ {\rm odd}}
      & 
      & \Multi{2}{$\Z_d$}
      & \Multi{2}{$\car{\frac{n+1}d \lambda_n}$}
      & \Multi{2}{$\omega(\lambda_n,\lambda_n)=\xi_d^k$, {\rm if}}
      \\
      $A_{n\geq 1}$
      & 
      & 
      & \multicolumn{3}{c}{\Multi{2}{$d\mid (n+1),~1\leq k\leq d$ {\rm and}}}
      \\\cline{2-3}
      $\pi_1=\Z_{n+1}$
      & \Multi{2}{$\ell$ {\rm even}}
      & $^*$
      & \multicolumn{3}{c}{\Multi{2}{$\gcd(n+1,d\ell,k\ell-\frac{n+1}d n)=1$}}
      \\
      &&&\multicolumn{3}{c}{}
      \\ \hline

      & \Multi{2}{$\ell$ {\rm odd}}
      & $1$
      & $\Z_1$
      & $\{0\}$
      & $\omega(0,0)=1$
      \\ \cline{3-6}
      &
      & $1$
      & $\Z_2$
      & $\car{\lambda_n}$
      & $\omega(\lambda_n,\lambda_n)=(-1)^{n-1}$
      \\ \cline{2-6}
      $B_{n\geq 2}$
      & \Multi{2}{$\ell\equiv 2 \mod 4$}
      & $2$
      & $\Z_2$
      & $\car{\lambda_n}$
      & $\omega(\lambda_n,\lambda_n)=\pm1$
      \\ \cline{3-6}
      $\pi_1=\Z_2$
      & 
      & $1$
      & $\Z_1$
      & $\{0\}$
      & $\omega(0,0)=1$, {\rm if} $n$ {\rm even}
      \\ \cline{2-6}
      & $\ell\equiv 0 \mod 4$
      & $2^*$
      & $\Z_2$
      & $\car{\lambda_n}$
      & $\omega(\lambda_n,\lambda_n)=\pm1$
      \\ \cline{3-6}
      & $\ell\neq 4$
      & $1^*$
      & $\Z_1$
      & $\{0\}$
      & $\omega(0,0)=1$, {\rm if} $n$ {\rm even}
      \\ \hline

    \enlargethispage{2cm}
      & \Multi{2}{$\ell$ {\rm odd}}
      & $1$
      & $\Z_1$
      & $\{0\}$
      & $\omega(0,0)=1$
      \\ \cline{3-6}
      &
      & $1$
      & $\Z_2$
      & $\car{\lambda_n}$
      & $\omega(\lambda_n,\lambda_n)=-1$
      \\ \cline{2-6}
      $C_{n\geq 3}$
      & \Multi{2}{$\ell\equiv 2\mod 4$}
      & $1$
      & $\Z_1$
      & $\{0\}$
      & $\omega(0,0)=1$
      \\ \cline{3-6}
      $\pi_1=\Z_2$
      &
      & $1$
      & $\Z_2$
      & $\car{\lambda_n}$
      & $\omega(\lambda_n,\lambda_n)=(-1)^{n-1}$
      \\ \cline{2-6}
      & $\ell\equiv 0 \mod 4$
      & $2^*$
      & $\Z_2$
      & $\car{\lambda_n}$
      & $\omega(\lambda_n,\lambda_n)=\pm1$
      \\ \cline{3-6}
      & $\ell\neq 4$
      & $1^*$
      & $\Z_1$
      & $\{0\}$
      & $\omega(0,0)=1$, {\rm if} $n$ {\rm even}
      \\ \hline

      & 
      & $1$
      & $\Z_1$
      & $\{0\}$
      & $\omega(0,0)=1$
      \\ \cline{3-6}
      &
      & $3$
      & $\Z_2$
      & $\car{\lambda}$
      & $\omega(\lambda,\lambda)= -1\!\!\!$
      \\ \cline{3-6}
      &
      & $6$
      & $\Z_2\neq\Z_2'$
      & $\car{\lambda},\car{\lambda'}$
      & $\omega(\lambda,\lambda')= 1\!\!\!$
      \\ \cline{3-6}
      & 
      & $1$
      & $\Z_2\times\Z_2$
      & $\car{\lambda_{n-1},\lambda_n}$
      & $\omega(\lambda_i,\lambda_j)=1$
      \\ \cline{3-6}
      & 
      & $1$
      & $\Z_2\times\Z_2$
      & $\car{\lambda_{n-1},\lambda_n}$
      & $\omega(\lambda_i,\lambda_j)=-1$
      \\ \cline{3-6}
      & $\ell$ {\rm odd}
      & \Multi{4}{$2$}
      & \Multi{4}{$\Z_2\times\Z_2$}
      & \Multi{4}{$\car{\lambda_{n-1},\lambda_n}$}
      & $\omega(\lambda_{n-1},\lambda_{n-1})= \pm1$
      \\
      &&&&
      & $\omega(\lambda_{n-1},\lambda_n)= 1$
      \\
      $D_{n\geq 4}$ 
      &&&&
      & $\omega(\lambda_{n},\lambda_{n-1})= 1$
      \\
      $n$ {\rm even} 
      &&&&
      & $\omega(\lambda_n,\lambda_n)= \mp1$
      \\ \cline{3-6}
      $\pi_1=\Z_2\times\Z_2$ 
      &
      & \Multi{4}{$2$}
      & \Multi{4}{$\Z_2\times\Z_2$}
      & \Multi{4}{$\car{\lambda_{n-1},\lambda_n}$}
      & $\omega(\lambda_{n-1},\lambda_{n-1})= -1$
      \\
      &&&&
      & $\omega(\lambda_{n-1},\lambda_n)= \pm1$
      \\
      &&&&
      & $\omega(\lambda_n,\lambda_{n-1})= \mp1$
      \\
      &&&&
      & $\omega(\lambda_n,\lambda_n)= -1$
      \\ \cline{2-6}
      & $\ell$ {\rm even}
      & $16^*$
      & $\Z_2\times\Z_2$
      & $\car{\lambda_{n-1},\lambda_n}$
      & $\omega(\lambda_i,\lambda_j)\in\{\pm1\}$
      \\ \hline

      & \Multi{3}{$\ell$ {\rm odd}}
      & $1$
      & $\Z_1$
      & $\{0\}$
      & $\omega(0,0)=1$
      \\ \cline{3-6}
      $D_{n\geq 5}$ 
      &
      & $1$
      & $\Z_2$
      & $\car{2\lambda_n}$
      & $\omega(2\lambda_n,2\lambda_n)=-1$
      \\ \cline{3-6}
      $n$ {\rm odd}
      &
      & $2$
      & $\Z_4$
      & $\car{\lambda_n}$
      & $\omega(\lambda_n,\lambda_n)=\pm1$
      \\ \cline{2-6}
      $\pi_1=\Z_4$
      & $\ell\equiv2\mod 4$
      & $4^*$
      & $\Z_4$
      & $\car{\lambda_n}$
      & $\omega(\lambda_n,\lambda_n)=c, ~c^4=1$ 
      \\ \cline{2-6}
      & $\ell\equiv0\mod4$
      & $4^*$
      & $\Z_4$
      & $\car{\lambda_n}$
      & $\omega(\lambda_n,\lambda_n)=c, ~c^4=1$
      \\ \hline

      & \Multi{2}{$\ell$ {\rm odd}, $3\nmid \ell$}
      & $1$
      & $\Z_1$
      & $\{0\}$
      & $\omega(0,0)=1$
      \\ \cline{3-6}
      &
      & $2$
      & $\Z_3$
      & $\car{\lambda_6}$
      & $\omega(\lambda_6,\lambda_6)= 1,\exp(\frac{2\pi i}3)$
      \\ \cline{2-6}
      $E_6$
      & \Multi{2}{$\ell$ {\rm even}, $3\nmid \ell$}
      & $1^*$
      & $\Z_1$
      & $\{0\}$
      & $\omega(0,0)=1$
      \\ \cline{3-6}
      $\pi_1=\Z_3$
      &
      & $2^*$
      & $\Z_3$
      & $\car{\lambda_6}$
      & $\omega(\lambda_6,\lambda_6)= 1,\exp(2\frac{2\pi i}3)$
      \\ \cline{2-6}
      & $\ell$ {\rm odd}, $3\mid \ell$
      & $3$
      & $\Z_3$
      & $\car{\lambda_6}$
      & $\omega(\lambda_6,\lambda_6)= c, ~c^3=1$
      \\ \cline{2-6}
      & $\ell$ {\rm even}, $3\mid \ell$
      & $3^*$
      & $\Z_3$
      & $\car{\lambda_6}$
      & $\omega(\lambda_6,\lambda_6)= c, ~c^3=1$
      \\ \hline

      \Multi{2}{$E_7$}
      & \Multi{2}{$\ell$ {\rm odd}}
      & $1$
      & $\Z_1$
      & $\{0\}$
      & $\omega(0,0)=1$
      \\ \cline{3-6}
      \Multi{2}{$\pi_1=\Z_2$}
      &
      & $1$
      & $\Z_2$
      & $\car{\lambda_7}$
      & $\omega(\lambda_7,\lambda_7)=1$
      \\ \cline{2-6}
      & $\ell$ {\rm even}
      & $2^*$
      & $\Z_2$
      & $\car{\lambda_7}$
      & $\omega(\lambda_7,\lambda_7)=\pm1$
      \\ \hline

      $E_8$
      & $\ell$ {\rm odd}
      & $1$
      & $\Z_1$
      & $\{0\}$
      & $\omega(0,0)=1$
      \\ \cline{2-6}
      $\pi_1=\Z_1$
      & $\ell$ {\rm even}
      & $1^*$
      & $\Z_1$
      & $\{0\}$
      & $\omega(0,0)=1$
      \\ \hline

      & $\ell$ {\rm odd}
      & $1$
      & $\Z_1$
      & $\{0\}$
      & $\omega(0,0)=1$
      \\ \cline{2-6}
      $F_4$
      & $\ell\equiv 2\mod 4$
      & $1$
      & $\Z_1$
      & $\{0\}$
      & $\omega(0,0)=1$
      \\ \cline{2-6}
      $\pi_1=\Z_1$
      & $\ell\equiv 0\mod 4$
      & \Multi{2}{$1^*$}
      & \Multi{2}{$\Z_1$}
      & \Multi{2}{$\{0\}$}
      & \Multi{2}{$\omega(0,0)=1$}
      \\ 
      &$\ell\neq 4$&&&& \\ \hline

      & $\ell$ {\rm odd}
      & \Multi{2}{$1$}
      & \Multi{2}{$\Z_1$}
      & \Multi{2}{$\{0\}$}
      & \Multi{2}{$\omega(0,0)=1$}
      \\ 
      $G_2$
      &$\ell\neq 3$
      &&&& \\ \cline{2-6}
      $\pi_1=\Z_1$
      & $\ell$ {\rm even}
      & \Multi{2}{$1^*$}
      & \Multi{2}{$\Z_1$}
      & \Multi{2}{$\{0\}$}
      & \Multi{2}{$\omega(0,0)=1$}
      \\ 
      &$\ell\neq 4,6$
      &&&& \\ \hline

    \caption{Solutions for $R_0$-matrices}
    \label{tbl:Solutions}
    \end{longtable}
  \end{thm}

 \setcounter{section}{0}

 The cases $B_n,C_n,F_4$, $\ell=4$ and $G_2$, $\ell=3,6$ and $\ell=4$
 respectively, can be obtained in the table for $A_1^{\times n},D_{n},
 D_4$, and again $A_2$ and $A_3$ respectively (cf. \cite{Len14} for details).

 Note, that Lusztig's $R$-matrix for $\Lambda=\Lambda_R$ correspond to the
 case $H=\Z_1$ and $\omega=1$. The known quadratic extension for $\Sl_2$ 
 is the case $A_1$ with $H=\Z_2$ in the example below.

 \begin{Rem}
   We indicate in which sense our results are \emph{not} complete:
   \begin{itemize}
     \item Technically, one could even allow $\Lambda_R\subset \Lambda\subset
       \Lambda_W^\vee$, but then one would loose the topological interpretation as
       different choices of a Lie group associated to $\g$.
     \item Our additional assumption \ref{ass:Lambda'} on the considered quotients 
       $\Lambda'\subset\Cent^q(\Lambda_W)\cap\Lambda_R$ was chosen 
       to simplify calculations and prove uniqueness. In general
       $\Lambda'\in\Cent^q(\Lambda)$ would suffice (and could yield more
       solutions), but one would have to deal with possible $2$-cocycles in
       $H^2(\Lambda/\Lambda',\pi_1)$ in Lemma \ref{g-eq}. 
   \end{itemize}
 \end{Rem}
 \begin{que}
  Are \emph{all} $R$-matrices of $u_q(\g)$ given by Lusztig's ansatz and hence in our list?
 \end{que}
 \begin{que}
  Which $R$-matrices above give rise to \emph{equivalent} braided tensor categories?
 \end{que}
 \begin{que}
  Which $R$-matrices in this article are \emph{factorizable}
  an give hence rise to (non-semisimple) modular tensor categories? What are results for 
  other Nichols algebras?
 \end{que}
 \begin{Expl}
   For $\g=\Sl_2$ with root system $A_1$ the fundamental group is
   $\pi_1=\Z_2$. Let $\alpha$ be the simple root, generating the root lattice
   $\Lambda_R$, and $\lambda=\frac12\alpha$ the fundamental dominant weight,
   generating the weight lattice $\Lambda_W$. We will give the $R$-matrices
   for the quantum groups $u=u_q(\g,\Lambda,\Lambda')$ for $\ell$-th root of
   unity $q$ and lattices $\Lambda_R\subset\Lambda\subset\Lambda_W$ and
   $\Lambda'=\Lambda_R^{[\ell]}$, which equals in the simply laced case
   $\ell\Lambda_R$.

   The quasi $R$-matrix $\Theta$ (see Theorem \ref{Thm:Theta}) depends only on
   the root lattice and exists in $u^+\otimes u^-$ with Borel parts $u^{\pm}$,
   generated by $E_{\alpha},F_{\alpha}$. With
   $\ell_{\alpha}=\ell/\gcd(\ell,2d_{\alpha})=\ell/\gcd(\ell,2)$ we have
   \begin{equation*}
     \Theta = \sum_{k=0}^{\ell_{\alpha-1}}(-1)^k \frac{(q-q^{-1})^k}{[k]_q!} 
     q^{-k(k-1)/2}
     E_{\alpha}^k \,\otimes\, F_{\alpha}^k
     \quad{\rm and}\quad
     \bar\Theta = \sum_{k=0}^{\ell_{\alpha-1}}\frac{(q-q^{-1})^k}{[k]_q!}
     q^{k(k-1)/2}
     E_{\alpha}^k \,\otimes\, F_{\alpha}^k,
   \end{equation*}
   with $q$-factorial $[k]_q!$.
   The toral part $R_0$-is given by
   \begin{equation*}
     R_0 = \frac {1}{|\Lambda/\Lambda_R^{[\ell]}|} 
     \sum_{\mu,\nu\in \Lambda/\Lambda'}
     q^{-(\mu,\nu)} \, \omega(\bar\mu,\bar\nu) \, 
     K_{\mu}\,\otimes\, K_{\nu},
   \end{equation*}
   for $H$ and $\omega\colon H\times H\to\C^{\times}$ as in Table
   \ref{tbl:Solutions}. The
   possible solutions depend on $\ell$. We now check the condition
   $\gcd(2,d\ell,k\ell-2/d)=1$ from the theorem
   above ($n=1$ and $d=1,2$). For odd $\ell$, we get the following
   solutions by Theorem \ref{Thm:solutions}:
   \begin{align*}
     H&=\Z_1,\qquad \omega\colon\Z_1\times\Z_1\to \C^{\times},~
     \omega(0,0)=1,\\
     H&=\Z_2,\qquad \omega\colon\Z_2\times\Z_2\to \C^{\times},~
     \omega(\lambda,\lambda)=1.
   \end{align*}
   For even $\ell$ the solution for $H=\Z_1$, i.e. for $\Lambda=\Lambda_R$,
   does not exist (since $2\mid \ell$ and $2\mid (\ell-2)$), rather we
   get both possible solutions on the full support $H=\Z_2$:
   \begin{align*}
     H&=\Z_2,\qquad \omega\colon\Z_2\times\Z_2\to \C^{\times},~
     \omega(\lambda,\lambda)=\pm1.
   \end{align*}
   In these cases, the $R$-matrices are explicitly given by
   \begin{align*}
     R 
     &= \frac {1}{2\ell} \sum_{k=0}^{\ell_{\alpha}-1}
        \sum_{i,j=0}^{2\ell-1} \frac{(q-q^{-1})^k}{[k]_q!}
        q^{k(k-1)/2+k(j-i)-\frac{ij}2} \, (\pm1)^{ij} \, 
       E_{\alpha}^k K_{\lambda}^i  \,\otimes\, F_{\alpha}^k K_{\lambda}^j 
   \end{align*}
 \end{Expl}

 \paragraph{\bf Acknowledgements.} The first author is supported by the DFG Research 
 Training Group 1670. We thank Christoph Schweigert for
 several helpful discussions.

 \numberwithin{thm}{section}
 \renewcommand{\thethm}{\arabic{section}.\arabic{thm}}

\section{Preliminaries}\label{prelim}

  At first, we fix a convention.
  \begin{conv}\label{conv}
      In the following, $q$ is an $\ell$-th root of unity.  We fix
      $q=\exp(\frac{2\pi i}{\ell})$ and for $a\in\R$ we set
      $q^a=\exp(\frac{2\pi ia}{\ell})$, $\ell>2$.
  \end{conv}

  \subsection{Lie Theory}
    Let $\g$ be a finite-dimensional, semisimple complex Lie algebra with
    simple roots $\alpha_i$, indexed by $i\in I$, $|I|=n$, and a
    set of positive roots $\Phi^{+}$. Denote the \emph{Killing form} by
    $(-,-)$, normalized such that $(\alpha,\alpha)=2$ for the short roots
    $\alpha$. The \emph{Cartan matrix} is given by
    \begin{equation*}
      a_{ij} = \car{\alpha_i,\alpha_j} 
             = 2\frac{(\alpha_i,\alpha_j)}{(\alpha_i,\alpha_i)}.
    \end{equation*}
    For a root $\alpha$ we call $d_\alpha:=(\alpha,\alpha)/2$ with
    $d_{\alpha}\in\{1,2,3\}$. Especially, $d_i:=d_{\alpha_i}$ and in this
    notation $(\alpha_i,\alpha_j)=d_ia_{ij}$.
    The fundamental dominant weights $\lambda_i,~i\in I$, are given by
    the condition $2(\alpha_i,\lambda_j)/(\alpha_i,\alpha_i)=\delta_{ij}$,
    hence the Cartan matrix expresses the change of basis from roots to weights.

    \begin{defi}
      The \emph{root lattice} $\Lambda_R=\Lambda_R(\g)$ of the Lie algebra
      $\g$ is the abelian group with rank
      $\mathrm{rank}(\Lambda_R)=\mathrm{rank}(\g)=|I|$, generated by the
      simple roots $\alpha_i$, $i\in I$. 
    \end{defi}

    \begin{defi}
      The \emph{weight lattice} $\Lambda_W=\Lambda_W(\g)$ of the Lie algebra
      $\g$ is the abelian group with rank
      $\mathrm{rank}(\Lambda_W)=\mathrm{rank}(\g)$, generated by the
      fundamental dominant weights $\lambda_i$, $i\in I$.   
    \end{defi} 

    The Killing form induces an integral pairing of abelian groups, turning
    $\Lambda_R$ into an \emph{integral lattice}. 
    It is standard fact of Lie theory (cf. \cite{Hum72}, Section 13.1) that
    the root lattice is contained in the weight lattice. 

    \begin{defi}\label{elllattice}
      Let $\Lambda_R$, $\Lambda_W$ the root, resp. weight, lattice of the Lie
      algebra $\g$ with generators $\alpha_i$, resp. $\lambda_i$, for $i\in
      I$.
      \begin{enumerate}[(i)]
        \item 
          Following Lusztig, we define $\ell_i :=\ell/\gcd(\ell,2d_i)$, which
          is the order of $q^{2d_i}$, where $q$ is a primitive $\ell$-th root
          of unity. 
          More generally, we define for any root
          $\ell_{\alpha}:=\ell/\gcd(\ell,2d_{\alpha})$.
          For any positive integer $\ell$, the \emph{$\ell$-lattice} 
          $\Lambda_R^{(\ell)}$, resp.  $\Lambda_W^{(\ell)}$, is defined
          as 
          \begin{equation}
            \Lambda_R^{(\ell)} =\left\langle
            \ell_i\alpha_i,~ i\in I
            \right\rangle
            \quad\text{resp.}\quad
            \Lambda_W^{(\ell)} =\left\langle
            \ell_i\lambda_i,~ i\in I
            \right\rangle.
          \end{equation}
        \item For any positive integer $\ell$, the lattice
          $\Lambda_R^{[\ell]}$, resp.  $\Lambda_W^{[\ell]}$, is defined as
          \begin{equation}
            \Lambda_R^{[\ell]} =\left\langle
            \frac{\ell}{\gcd(\ell,d_i)}\alpha_i,~ i\in I
            \right\rangle
            \quad\text{resp.}\quad
            \Lambda_W^{[\ell]} =\left\langle
            \frac{\ell}{\gcd(\ell,d_i)}\lambda_i,~ i\in I
            \right\rangle.
          \end{equation}
      \end{enumerate}
    \end{defi}

    \begin{defi}
      For $\Lambda_1, \Lambda_2\subset\Lambda_W$ with
      $\Lambda_2\subset\Lambda_1$ we define $\Cent^q_{\Lambda_1}(\Lambda_2)=
      \{\eta\in\Lambda_1~|
      ~(\eta,\lambda)\in\ell\Z~\forall\lambda\in\Lambda_2\}$.
      In the situation $\Lambda_1=\Lambda_W$ we simply write 
      $\Cent^q_{\Lambda_W}(\Lambda_2)=\Cent^q(\Lambda_2)$.
    \end{defi}

    Especially, the set $\car{K_{\eta} ~|~ \eta\in\Cent^q(\Lambda_R)}$ consists
    of the central group elements of the quantum group $U_q(\g,\Lambda_W)$,
    cf. Section \ref{qgrp}.

    \begin{lem}\label{CentR}
      For a Lie algebra $\g$ we have
        $\Cent^q(\Lambda_R)=\Lambda_{W}^{[\ell]}$.
      We call the elements of $\Cent^q(\Lambda_R)$ \emph{central weights}.
    \end{lem}

    \begin{proof}
      Let $\lambda=\sum_{j\in I}m_j\lambda_j\in\Lambda_W$ with fundamental
      weights $\lambda_i$. For a simple root $\alpha_i$ we have
      $(\alpha_i,\lambda) = (\alpha_i,\sum_{j\in I}m_j \lambda_j) = d_im_i$.
      Thus, $\lambda$ is central weight if $\ell\mid d_im_i$ for all $i$,
      hence 
        $(\ell/\gcd(\ell,d_i))\mid m_i$
      for all $i$.
    \end{proof}

    The same calculation gives the following lemma.

    \begin{lem}\label{CentW}
      For a Lie algebra $\g$ we have
        $\Cent^q(\Lambda_W)\cap\Lambda_R=\Lambda_{R}^{[\ell]}$.
    \end{lem}

  \subsection{Quantum groups}\label{qgrp}

    For a finite-dimensional complex simple Lie algebra $\g$, lattices
    $\Lambda, \,\Lambda'$ with $\Lambda_R \subset \Lambda \subset \Lambda_W$
    and $2\Lambda_R^{(\ell)}\subset
    \Lambda'\subset\Cent^q(\Lambda_W)\cap\Lambda_R$, and a primitive $\ell$-th
    root of unity $q$, we aim to define the finite-dimensional quantum group
    $u_q(\g,\Lambda,\Lambda')$, also called small quantum group.  We construct
    $u_q(\g,\Lambda,\Lambda')$ by using rational and integral forms of the
    deformed universal enveloping algebra $U_q(\g)$ for an indeterminate
    $q$. In the following we give the definitions of the quantum groups,
    following the lines of \cite{Len14}. The different choices of
    $\Lambda$ are already in \cite{Lus93}, Sec. 2.2. We shall give a
    dictionary to translate Lusztig's notation to the one used here.

    \begin{defi}\label{quantumBinom}
      For $q\in\C^{\times}$ or $q$ an indeterminate and $n\leq k\in\N_0$ we
      define
      \begin{equation*}
        [n]_q \df \frac{q^n-q^{-n}}{q-q^{-1}} \qquad
        [n]_q! \df [1]_1 [2]_q\dots [n]_q \qquad
        \begin{bmatrix} n \\ k \end{bmatrix}_q \df
        \begin{cases} \frac{[n]_q!}{[k]_q! [n-k]_q!}, & 0\leq k\leq n, \\
                      0,                           & \text{else}.
        \end{cases}
      \end{equation*}
    \end{defi}

    \begin{defi}\label{rationalForm}
      Let $q$ be an indeterminate.
      For each abelian group $\Lambda$ with $\Lambda_R\subset \Lambda\subset
      \Lambda_W$ we define the \emph{rational form} $U_q^{\Q(q)}(\g,\Lambda)$
      over the ring of rational functions $\kk=\Q(q)$ as follows:

      As algebra, let $U_q^{\Q(q)}(\g,\Lambda)$ be generated by the group ring
      $\kk[\Lambda]$, spanned by $K_{\Lambda}$, $\lambda\in\Lambda$, and
      additional generators $E_{\alpha_i},\, F_{\alpha_i}$, for each simple
      root $\alpha_i,\, i\in I$, with relations:
      \begin{gather}
        K_{\lambda}E_{\alpha_i}K_{\lambda}^{-1} =
        q^{(\lambda,\alpha_i)}E_{\alpha_i},
        \label{uq2}\\
        K_{\lambda}F_{\alpha_i}K_{\lambda}^{-1} =
        q^{-(\lambda,\alpha_i)}F_{\alpha_i},
        \label{uq3}\\
        E_{\alpha_i}F_{\alpha_j}-F_{\alpha_j}E_{\alpha_i} =\delta_{ij}
        \frac{K_{\alpha_i}-K_{\alpha_i}^{-1}}{q_{\alpha_i} -
        q_{\alpha_i}^{-1}}, \label{uq4}
      \end{gather}
      and Serre relations for any $i\neq j\in I$
      \begin{gather}
        \sum_{r=0}^{1-a_{ij}}(-1)^r \begin{bmatrix} 1-a_{ij}\\r
        \end{bmatrix}_{q_i} E_{\alpha_i}^{1-a_{ij}-r} E_{\alpha_j}
        E_{\alpha_i}^r = 0,\label{qSerreE}\\
        \sum_{r=0}^{1-a_{ij}}(-1)^r \begin{bmatrix} 1-a_{ij}\\r
        \end{bmatrix}_{\bar q_i} F_{\alpha_i}^{1-a_{ij}-r} F_{\alpha_j}
        F_{\alpha_i}^r = 0,\label{qSerreF}
      \end{gather}
      where $\bar q\df q^{-1}$, the quantum binomial coefficients are defined
      in Definition \ref{quantumBinom} and by definition
      $q^{(\alpha_i,\alpha_j)}=(q^{d_i})^{a_{ij}}=q_i^{a_{ij}}$.

      As a coalgebra, let the coproduct $\Delta$, the counit $\varepsilon$ and
      the antipode $S$ be defined on the group-Hopf-algebra $\kk[\Lambda]$ as
      usual
      \begin{equation*}
        \Delta(K_{\lambda}) = K_{\lambda}\otimes K_{\lambda}, \qquad
        \varepsilon(K_{\lambda}) = 1, \qquad
        S(K_{\lambda}) = K_{\lambda}^{-1} = K_{-\lambda},
      \end{equation*}
      and on the generator $E_{\alpha_i}, F_{\alpha_i}$, for each simple
      root $\alpha_i$, $i\in I$ as follows
      \begin{align*}
        \Delta(E_{\alpha_i})=E_{\alpha_i}\otimes K_{\alpha_i} + 1\otimes
        E_{\alpha_i},& 
        \quad \Delta(F_{\alpha_i})= F_{\alpha_i}\otimes 1 + K_{\alpha_i}^{-1}
        \otimes F_{\alpha_i},\\
        \varepsilon(E_{\alpha_i})= 0,& 
        \quad \varepsilon(F_{\alpha_i}) = 0,\\
        S(E_{\alpha_i})= -E_{\alpha_i}K_{\alpha_i}^{-1},&
        \quad S(F_{\alpha_i})=-K_{\alpha_i}F_{\alpha_i}.
      \end{align*}
    \end{defi}

      This is a Hopf algebra over the field $\kk=\Q(q)$.
      Moreover, we have a triangular decomposition: Consider the subalgebras
      $U_q^{\Q(q),+}$, generated by the $E_{\alpha_i}$, and $U_q^{\Q(q),-}$,
      generated by the $F_{\alpha_i}$, and $U_q^{\Q(q),0}=\kk[\Lambda]$,
      spanned by the $K_{\lambda}$. Then the multiplication in
      $U_q^{\Q(q)}=U_q^{\Q(q)}(\g,\Lambda)$ induces an isomorphism of vector
      spaces:
      \begin{equation*}
        U_q^{\Q(q),+} \otimes U_q^{\Q(q),0} \otimes U_q^{\Q(q),-} 
        \overset{\cong}{\longrightarrow}U_q^{\Q(q)}.
      \end{equation*}
    
    \begin{defi}
      The so-called \emph{restricted integral form} $U_q^{\Z[q,q^{-1}]}
      (\g,\Lambda)$ is generated as a $\Z[q,q^{-1}]$-algebra by $\Lambda$ and
      the following elements in $U_q^{\Q(q),\pm}(\g,\Lambda)$, called
      \emph{divided powers}:
      \begin{equation*}
        E_{\alpha}^{(r)} \df \frac{E_{\alpha}^r} 
        {\prod_{s=1}^r
        [s]_{q_\alpha}}
        \quad
        F_{\alpha}^{(r)} \df \frac{F_{\alpha}^r} 
        {\prod_{s=1}^r
        [s]_{\bar q_\alpha}}
        \quad \text{for all }\alpha\in\Phi^{+},r>0,
      \end{equation*}
      and by the following elements in $U_q^{\Q(q)}(\g,\Lambda)^0$:
      \begin{equation*}
        K_{\alpha_i}^{(r)} = 
        \begin{bmatrix} K_{\alpha_i};0 \\ r \end{bmatrix} \df
        \prod_{s=1}^r 
        \frac{K_{\alpha_i}q_{\alpha_i}^{1-s}-K_{\alpha_i}^{-1}q_{\alpha_i}^{s-1}}
             {q_{\alpha_i}^s - q_{\alpha_i}^{-s}},
        \qquad i\in I.
      \end{equation*}
    \end{defi}

    These definitions can also be found in Lusztig's book \cite{Lus93}. In order
    to translate Lusztig's notation to the one used here, one has to match the
    terms in the following way

    \begin{center}
    \begin{tabular}{c|c}
      Lusztig's notation & notation used here \\ \hline
      Index set $I$ & simple roots $\{\alpha_i\mid i\in I\}$ \\
      $X$ & root lattice $\Lambda_R$ \\
      $Y$ & lattice $\Lambda_R\subset \Lambda\subset \Lambda_W$ \\
      $i'\in X$ & ${\alpha_i}$ \\
      $i\in Y$ & $\frac{\alpha_i}{d_{\alpha_i}}=\alpha_i^{\vee}$ coroot \\
      $i\cdot j$, $i,j\in \Z[I]$ & $(\alpha_i,\alpha_j)$ \\
      $\car{i,j'}=2\frac{i\cdot j}{i\cdot i},~ i\in Y, j'\in X$ 
      & $\car{\alpha_i,\alpha_j}$ \\
      $K_i$ & $K_{\alpha_i^{\vee}}$ \\
      $\tilde K_i=K_{\frac{i\cdot i}2 i}$ & $K_{\alpha_i}$
    \end{tabular} 
    \end{center}

    We now define the \emph{restricted specialization} $U_q(\g,\Lambda)$.
    Here, we specialize $q$ to a specific choice $q\in\C^{\times}$.

    \begin{defi}\label{Uqg}
      The infinite-dimensional Hopf algebra $U_q(\g,\Lambda)$ is defined by
      \begin{equation*}
        U_q(\g,\Lambda)\df U_q^{\Z[q,q^{-1}]}(\g,\Lambda)
        \otimes_{\Z[q,q^{-1}]} \C_q,
      \end{equation*}
      where $\C_q=\C$ with the $\Z[q,q^{-1}]$-module structure defined by the
      specific value $q\in\C^{\times}$.
    \end{defi}

    From now on, $q$ will be a primitive $\ell$-th root of unity. We choose
    explicitly $q=\exp(\frac{2\pi i}{\ell})$, see Convention \ref{conv}.

    \begin{defi}\label{uqg}
      Let $\g$ be a finite-dimensional complex simple Lie algebra with root
      system $\Phi$ and assume ${\rm ord}(q^2)>d_\alpha$ for all
      $\alpha\in\Phi$. For lattices $\Lambda,\Lambda'$ with $\Lambda_R\subset
      \Lambda\subset\Lambda_W$ and $2\Lambda_R^{(\ell)}\subset \Lambda'\subset
      \Cent^q(\Lambda_W)\cap\Lambda_R$, we define the \emph{small quantum group}
      $u_q(\g,\Lambda,\Lambda')$ as the algebra $U_q(\g,\Lambda)$ from
      Definition \ref{Uqg}, generated by $K_{\lambda}$ for $\lambda\in\Lambda$
      and $E_{\alpha},F_{\alpha}$ with $\ell_{\alpha}>1$, $\alpha\in\Phi^+$ not
      necessarily simple, together with the relations
      \begin{equation*}
        E_{\alpha}^{\ell_{\alpha}}=0, \quad
        F_{\alpha}^{\ell_{\alpha}}=0 \quad\text{and}\quad
        K_{\lambda} = 1 ~ \text{for } \lambda\in\Lambda'.
      \end{equation*}
      The coalgebra structure is again given as in Definition
      \ref{rationalForm}.
      This is a finite dimensional Hopf algebra of dimension 
      \begin{equation*}
        |\Lambda/\Lambda'|\prod_{\alpha\in\Phi^+,~ \ell_{\alpha}>1}
        \ell_{\alpha}^2.
      \end{equation*}
    \end{defi}

    The fact, that this gives a Hopf algebra for
    $\Lambda'=2\Lambda_R^{(\ell)}$ is in Lusztig, \cite{Lus90}, Sec. 8. 

    We fix the assumption on $\Lambda'$.

    \begin{ass}\label{ass:Lambda'}
      We assume for the sublattice $\Lambda'\subset\Lambda_W$ in the following that
      \begin{equation*}
         2\Lambda_R^{(\ell)}\subset \Lambda'\subset
         \Cent^q(\Lambda_W)\cap\Lambda_R
      \end{equation*}
    \end{ass}
  \subsection{$R$-matrices}
    \begin{defi}
      A Hopf algebra $H$ is called \emph{quasitriangular} if there exists an
      invertible element $R\in H\otimes H$ such that
      \begin{align}
        \Delta^{op}(h)          &= R\Delta(h) R^{-1}, \label{rmatrix}\\
        (\Delta \otimes \Id)(R) &= R_{13}R_{23}, \label{co1}\\
        (\Id \otimes \Delta)(R) &= R_{13}R_{12}, \label{co2}
      \end{align}
      with $\Delta^{op}(h)=\tau\circ \Delta(h)$, where $\tau: H\otimes
      H\longrightarrow H\otimes H, ~a\otimes b \longmapsto b\otimes a$ and
      $R_{12}=R\otimes 1,~ R_{23} = 1 \otimes R,~ R_{13} = (\tau \otimes
      \Id)(R_{23})=(\Id\otimes \tau)(R_{12}) \in H^{\otimes 3}$. Such an
      element is called an \emph{$R$-matrix of $H$}.
    \end{defi}

\section{Ansatz for $R$}\label{ansatz}
  \subsection{Quasi-$R$-matrix and Cartan-part}
  The goal of this paper is to construct new families of $R$-matrices for small
  quantum groups and certain extensions (see Def. \ref{uqg}). Our starting point
  is Lusztig's ansatz in \cite{Lus93}, Sec. 32.1, for a universal $R$-matrix
  of $U_q(\g,\Lambda)$. This ansatz has been translated by M\"uller in his
  Dissertation \cite{Mue98a}, resp. in \cite{Mue98b}, for small quantum
  groups, which we will use in the following. Note, that this ansatz has been
  successfully generalized to general diagonal Nichols algebras in \cite{AY13}.

  For a finite-dimensional, semisimple complex Lie algebra $\g$, an $\ell$-th
  root of unity $q$ and lattices $\Lambda,\Lambda'$ as in Section \ref{qgrp},
  we write $u=u_q(\g,\Lambda,\Lambda')$.  
  Let ${~}\bar{~}\colon u\to \bar u$ be the $\Q$-algebra isomorphism defined
  by $q\mapsto q^{-1},~E_{\alpha_i}\mapsto E_{\alpha_i},~ F_{\alpha_i}\mapsto
  F_{\alpha_i}$, $i\in I$, and $K_{\lambda}\mapsto K_{-\lambda}$,
  $\lambda\in\Lambda$. Then the map ${~}\bar{}\,\otimes\,\bar{}\,\colon
  u\otimes u\rightarrow \bar u\otimes \bar u$ is a well-defined $\Q$-algebra
  isomorphism and we can define a $\Q(q)$-algebra morphism $\bar\Delta\colon
  u\rightarrow u\otimes u$ given by $\bar\Delta (x) = \overline{\Delta(\bar
  x)}$ for all $x\in U$. We have in general $\bar\Delta\neq\Delta$.

  Assume in the following, that
  \begin{equation}\label{eq:ell-Cond}
    \ell_i> 1 \text{ for all } i\in I,\text{ and }
    \ell_i> -\car{\alpha_i,\alpha_j} \text{ for all }i,j\text{ with }i\neq j.
  \end{equation}

    \begin{thm}[\cite{Len14}]\label{rem:AltRootSys}
      For a root system $\Phi$ of a finite-dimensional simple complex Lie
      algebra and an $\ell$-th root of unity $q$, the condition
      \eqref{eq:ell-Cond} fails only in the following cases $(\Phi,\ell)$. In
      each case, the small quantum group $u_q(\g)$ is described by a different
      $\tilde\Phi$ fulfilling \eqref{eq:ell-Cond}, hence the present work also
      provides results for these cases by consulting the results for
      $\tilde\Phi$.

      \begin{center}
        \begin{tabular}{c||c|c|c|c|c|c}
          $\Phi$ & (all) & $B_n$ & $C_n$ & $F_4$ & $G_2$ & $G_2$ \\ \hline
          $\ell$ & $1,2$ & $4$   & $4$   & $4$   & $3,6$ & $4$   \\
          \hline\hline 
          &&&&&&\\[-7pt]
          $\tilde\Phi$ & (empty) &
          $\underset{n\text{-times}}{\underbrace{A_1\times\ldots\times A_1}}$ &
          $D_{n}$ & $D_4$ & $A_2$ & $A_3$
        \end{tabular}
      \end{center}
    \end{thm}

    The following theorem is essentially in \cite{Lus93}. Note that the roles
    of $E,F$ will be switched in our article to match the usual convention: 
  
  \begin{thm}[cf. \cite{Mue98b}, Thm. 8.2]\label{Thm:Theta}
    (a) There is a unique family of elements $\Theta_\nu\in u_{\nu}^+\otimes
    u_{\nu}^-$, $\nu\in\Lambda_R$, such that $\Theta_0=1\otimes 1$ and
    $\Theta=\sum_{\nu}\Theta_{\nu}\in u\otimes u$ satisfies
    $\Delta(x)\Theta = \Theta \bar\Delta(x)$ for all $x\in u$.

    (b) Let $B$ be a vector space-basis of 
    $u^+$, such that $B_{\nu}=B\cap u^+_{\nu}$ is a basis of
    $u^+_{\nu}$ for all $\nu$. Here, $u_{\nu}^+$ refers to the natural
    $\Lambda_R$-grading of $u^+$. Let $\{b^* ~|~ b\in B_{\nu}\}$ be the
    basis of $u_{\nu}^-$ dual to $B_{\nu}$ under the non-degenerate bilinear
    form $(\,\cdot\,,\,\cdot\,)\colon u^+\otimes u^-\to \C$.
    We have
    \begin{equation}
      \Theta_{\nu} \eq (-1)^{{\rm tr}\, \nu} q_{\nu} \sum_{b\in B_{\nu}} b^+
      \otimes b^{*-} \in u_{\nu}^+ \otimes u_{\nu}^-,
    \end{equation}
    where $q_{\nu}=\prod_i q_i^{\nu_i}$, ${\rm tr}\,\nu=\sum_i\nu_i$ for
    $\nu=\sum_i\nu_i \alpha_i\in\Lambda_R$.
  \end{thm}

  \begin{rem}
    \begin{enumerate}[(i)]
      \item The element $\Theta$ is called the \emph{Quasi-$R$-matrix} of
        $u=u_q(\g,\Lambda,\Lambda')$.
      \item Since the element $\Theta$ is unique, the expressions $\sum_{b\in
        B_{\nu}} b^+ \otimes b^{*-}$ in part (b) of the theorem are
        independent of the actual choice of the basis $B$.
      \item For example, if $\g=A_1$, i.e. there is only one simple root
        $\alpha=\alpha_1$, and $E=E_{\alpha}$, $F=F_{\alpha}$. Thus we have
        \begin{equation*}
          \Theta \eq \sum_{n=0}^{\ell_{\alpha}-1} (-1)^n
          \frac{(q-q^{-1})^n}{[n]_q!} q^{-n(n-1)/2} E^n \otimes F^n.
        \end{equation*}
      \item  The Quasi-$R$-matrix $\Theta$ is invertible with inverse
        $\Theta^{-1} = \bar \Theta$, i.e. the expression one gets by changing
        all $q$ to $\bar q = q^{-1}$. 
    \end{enumerate}
  \end{rem}

  \begin{thm}[cf. \cite{Mue98b}, Theorem 8.11]\label{Thm:R0}
    Let $\Lambda'\subset\{\mu\in\Lambda~|~K_{\mu} \text{ central in }
    u_q(\g,\Lambda)\}$ be a subgroup of $\Lambda$, and $H_1, H_2$ be subgroups of
    $\Lambda/\Lambda'$, containing $\Lambda_R/\Lambda'$. In the following,
    $\mu,\mu_1,\mu_2\in H_1$ and $\nu,\nu_1,\nu_2\in H_2$.

    The element $R=R_0\bar\Theta$ with $R_0= \sum_{\mu,\nu} f(\mu,\nu)
    K_{\mu} \otimes K_{\nu}$ is an $R$-matrix for
    $u_q(\g,\Lambda,\Lambda')$, if and only if for all $\alpha\in
    \Lambda_R$ and $\mu,\nu$ the following holds:
    \begin{align}
      f(\mu+ \alpha, \nu) = q^{-(\nu, \alpha)} f(\mu,\nu),
      &\quad
      f(\mu, \nu+ \alpha) = q^{-(\mu, \alpha)} f(\mu,\nu),
      \label{f01}
      \\
      \sum_{\substack{\nu_1,\nu_2\in H_2\\\nu_1+\nu_2 = \nu}}
      f(\mu_1,\nu_1)f(\mu_2,\nu_2) = \delta_{\mu_1,\mu_2} f(\mu_1,\nu),
      &\quad
      \sum_{\substack{\mu_1,\mu_2\in H_1\\\mu_1+\mu_2 = \mu}}
      f(\mu_1,\nu_1)f(\mu_2,\nu_2) = \delta_{\nu_1,\nu_2} f(\mu,\nu_1), 
      \label{f02}
      \\
      \sum_{\mu} f(\mu,\nu) = \delta_{\nu,0},
      &\quad
      \sum_{\nu} f(\mu,\nu) = \delta_{\mu,0}.
      \label{f03}
    \end{align}
    Condition \ref{f03} follows from \ref{f01} and \ref{f02} if there exists
    $c\in\C$ such that $f(\mu,0)=f(0,\nu)=c$ for all $\mu,\nu$. There are
    conditions on the order of $q$: For all $\mu,\nu$ for which there exist
    $\tilde\mu,\tilde\nu$ such that $f(\mu,\tilde\nu)\neq 0$,
    $f(\tilde\mu,\nu)\neq0$ we have
    \begin{equation*}
      q^{2l_i\car{\mu,\alpha_i}}=q^{2l_i\car{\nu,\alpha_i}}=1.
    \end{equation*}
    If this condition is satisfied then $f$ is well-defined on the preimages
    of $H_1\times H_2$ under $\Lambda\to\Lambda/\Lambda'$.
    (In particular, this is the case under our assumption
    $\Lambda'\subset\Cent^q(\Lambda_W)$.)
  \end{thm}

  \subsection{A set of equations}
  \begin{lem}\label{g-eq}
    Let $\Lambda\subset\Lambda_W$ a sublattice and $\Lambda'\subset\Lambda$.
    Assume in addition, $\Lambda'\subset \Cent^q(\Lambda_W)$.
    \begin{enumerate}[(i)]
      \item Let $f:\Lambda/\Lambda'\times\Lambda/\Lambda' \to\C$, satisfying
        condition \eqref{f01} of Theorem \ref{Thm:R0}. Then 
        \begin{equation}
          g(\bar\mu,\bar\nu) \df |\Lambda_R/\Lambda'| q^{(\mu,\nu)}f(\mu,\nu),
        \end{equation}
        defines a function $\pi_1\times\pi_1\to \C$.
      \item If, in addition, $f$ satisfies conditions\eqref{f02}-\eqref{f03},
        the function $g$ in (i) satisfies the following equations:
        \begin{align}
          \begin{split}
            \sum_{\bar\nu_1+\bar\nu_2 = \bar\nu} \delta_{(\mu_2-\mu_1\in
            \Cent^q(\Lambda_R))}
            q^{(\mu_2-\mu_1,\bar\nu_1)} g(\bar\mu_1,\bar\nu_1)
            g(\bar\mu_2,\bar\nu_2) &= \delta_{\mu_1,\mu_2}
            g(\bar\mu_1,\bar\nu),
            \\
            \sum_{\bar\mu_1+\bar\mu_2 = \bar\mu} \delta_{(\nu_2-\nu_1\in
            \Cent^q(\Lambda_R))}
            q^{(\nu_2-\nu_1,\bar\mu_1)} g(\bar\mu_1,\bar\nu_1)
            g(\bar\mu_2,\bar\nu_2) &= \delta_{\nu_1,\nu_2}
            g(\bar\mu,\bar\nu_1), 
          \end{split}
          \label{g02}
        \end{align}
        \vspace{-0.4cm}
        \begin{align}
          \begin{split}
            \sum_{\bar\nu} \delta_{(\mu\in \Cent^q(\Lambda_R))} q^{-(\mu,\bar\nu)}
            g(\bar\mu,\bar\nu) &= \delta_{\mu,0},
            \\
            \sum_{\bar\mu} \delta_{(\nu\in \Cent^q(\Lambda_R))} q^{-(\nu,\bar\mu)}
            g(\bar\mu,\bar\nu) &= \delta_{\nu,0}.
          \end{split}
          \label{g03}
        \end{align}
        Here, the sums range over $\pi_1$ and expressions like $\delta_{(\mu\in
        \Cent^q(\Lambda_R))}$ equals $1$ if $\mu$ is a central weight and $0$
        otherwise.
    \end{enumerate}
  \end{lem}

  Before we proceed with the proof we will comment
  on the relevance of this equations and introduce a definition. 
  For a given Lie algebra $\g$ with root lattice $\Lambda_R$ and
  weight lattice $\Lambda_W$ the solutions of the
  $g(\bar\mu,\bar\nu)$-equations give solutions for an $R_0$ in the
  ansatz $R=R_0\bar\Theta$. Hence, we get possible $R$-matrices for the
  quantum group $u_q(\g,\Lambda_W,\Lambda')$. 

  We divide the equations in two types. 
  \begin{defi}
    For central weight $0$ we call the equations \eqref{g02}-\eqref{g03}
    \emph{group-equations}: 
    \begin{align*}
      g(\bar\mu,\bar\nu) &= \sum_{\bar\nu_1+\bar\nu_2=\bar\nu}
      g(\bar\mu,\bar\nu_1)g(\bar\mu,\bar\nu_2),\\
      g(\bar\mu,\bar\nu) &= \sum_{\bar\mu_1+\bar\mu_2=\bar\mu}
      g(\bar\mu_1,\bar\nu)g(\bar\mu_2,\bar\nu),\\
      1 &= \sum_{\bar\nu} g(0,\bar\nu),\\
      1 &= \sum_{\bar\mu} g(\bar\mu,0).
    \end{align*}
    For $\pi_1=\Lambda_W/\Lambda_R$ of order $n$ this gives us $2n^2+2$
    group-equations.
  
    For central weight $0\neq\zeta\in\Cent^q(\Lambda_R)/\Lambda'$, we call the
    equations \eqref{g02}-\eqref{g03} \emph{diamond-equations} (for reasons
    that will become transparent later): 
    \begin{align*}
      0 &= \sum_{\bar\nu_1+\bar\nu_2=\bar\nu}
      q^{(\zeta,\bar\nu_1)}g(\bar\mu,\bar\nu_1)g(\bar\mu+\bar\zeta,\bar\nu_2),\\
      0 &= \sum_{\bar\mu_1+\bar\mu_2=\bar\mu}
      q^{(\zeta,\bar\mu_1)}g(\bar\mu_1,\bar\nu)g(\bar\mu_2,\bar\nu+\bar\zeta),\\
      0 &= \sum_{\bar\nu} q^{-(\bar\nu,\zeta)}g(\bar\mu+\bar\zeta,\bar\nu),\\
      0 &= \sum_{\bar\mu} q^{-(\bar\mu,\zeta)}g(\bar\mu,\bar\nu+\bar\zeta).
    \end{align*}
    This gives up to $(|\mathrm{Cent}^{[\ell]}(\Lambda_R)/\Lambda'|-1)(2n^2+2)$
    diamond-equations.
  \end{defi}

  \begin{proof}[Proof of Lemma \ref{g-eq}]
    \begin{enumerate}[(i)]
      \item Since $\Lambda'\subset\Cent^q(\Lambda_W)$ we have
        $q^{(\Lambda_W,\Lambda')}=1$ and terms $q^{(\mu,\nu)}$ for
        $\mu,\nu\in\Lambda/\Lambda'$ do not depend on the residue class
        representatives modulo $\Lambda'$.
        We check that the function $g$ is well-defined. Let $\mu,\nu\in
        \Lambda$ and $\lambda'\in\Lambda'$.  Thus,
        \begin{align*}
        g(\mu+\lambda',\nu)
          &= |\Lambda_R/\Lambda'| q^{(\mu+\lambda',\nu)} f(\mu+\lambda',\nu)\\
          &= |\Lambda_R/\Lambda'| q^{(\mu+\lambda',\nu)} q^{-(\lambda',\nu)}
        f(\mu,\nu) \tag*{by eq. \eqref{f01}}\\
          &= |\Lambda_R/\Lambda'| q^{(\mu,\nu)} f(\mu,\nu) \\
          &= g(\mu,\nu),
        \end{align*}
        and analogously for $g(\mu,\nu+\lambda')$.
      \item  We consider equations \eqref{f02}. Let
        $\nu_i,\nu\in\Lambda/\Lambda'$ and write $\nu_i=\bar\nu_i+\alpha_i$
        and $\nu=\bar\nu+\alpha$ with
        $\bar\nu_i,\bar\nu\in\Lambda_W/\Lambda_R$ and
        $\alpha_i,\alpha\in\Lambda_R$, $i=1,2$. For the sum $\nu=\nu_1+\nu_2$
        we get $\bar\nu\equiv\bar\nu_1+\bar\nu_2$ in $\Lambda_W/\Lambda_R$,
        i.e. there is a cocycle $\sigma(\nu_1,\nu_2)\in\Lambda_R$ with
        $\bar\nu=\bar\nu_1+\bar\nu_2+\sigma(\nu_1,\nu_2)$ in $\Lambda_W$ and
        $\alpha=\alpha_1+\alpha_2-\sigma(\nu_1,\nu_2)$. We will write $\sigma$
        for $\sigma(\nu_1,\nu_2)$.
    \begin{align*}
      \sum_{\nu_1+\nu_2=\nu} &f(\mu_1,\nu_1) f(\mu_2,\nu_2) \\
      &= \sum_{\nu_1+\nu_2=\nu} q^{-(\mu_1,\nu_1)+ (\bar\mu_1,\bar\nu_1) -
          (\mu_2,\nu_2)+ (\bar\mu_2,\bar\nu_2)} f(\bar\mu_1,\bar\nu_1)
          f(\bar\mu_2,\bar\nu_2)\\
      &= \sum_{\bar\nu_1+\bar\nu_2=\bar\nu}
         \sum_{\alpha_1+\alpha_2=\alpha+\sigma}
         q^{-(\mu_1,\bar\nu_1) - (\mu_1,\alpha_1) + (\bar\mu_1,\bar\nu_1)}
         q^{-(\mu_2,\bar\nu_2) - (\mu_2,\alpha_2) + (\bar\mu_2,\bar\nu_2)}
          f(\bar\mu_1,\bar\nu_1) f(\bar\mu_2,\bar\nu_2)\\
      &= \sum_{\bar\nu_1+\bar\nu_2=\bar\nu}
         q^{-(\mu_1,\bar\nu_1) + (\bar\mu_1,\bar\nu_1)
            -(\mu_2,\bar\nu_2) + (\bar\mu_2,\bar\nu_2)}
          f(\bar\mu_1,\bar\nu_1) f(\bar\mu_2,\bar\nu_2)
         \sum_{\alpha_1+\alpha_2=\alpha+\sigma}
         q^{- (\mu_1,\alpha_1) - (\mu_2,\alpha_2)} \tag{$\ast$}
    \end{align*}
    Firstly, we consider the second sum over the roots ($\mu_1,\mu_2$
    are fixed).
    \begin{align*}
      \sum_{\alpha_1+\alpha_2=\alpha+\sigma}
      q^{- (\mu_1,\alpha_1) - (\mu_2,\alpha_2)}
      &= \sum_{\alpha_1\in\Lambda_R/\Lambda'}
      q^{- (\mu_1,\alpha_1) - (\mu_2,\alpha+\sigma-\alpha_1)} \\
      &= q^{-(\mu_2,\alpha+\sigma)} \sum_{\alpha_1\in\Lambda_R/\Lambda'}
      q^{(\mu_2-\mu_1,\alpha_1)}
    \end{align*}
    The last sum equals $|\Lambda_R/\Lambda'|$ iff
    $\ell\mid(\mu_2-\mu_1,\alpha_1)$ for all $\alpha_1\in\Lambda_R/\Lambda'$,
    i.e. $\mu_2-\mu_1\in\Cent^q(\Lambda_R)$, and $0$ otherwise.
    Hence, with $C=|\Lambda_R/\Lambda'| \cdot \delta_{(\mu_2-\mu_1\in
    \mathrm{Cent}^{[\ell]}(\Lambda_R))}$, the sum $(\ast)$ simplifies to
    \begin{align*}
      C\,\cdot&\sum_{\bar\nu_1+\bar\nu_2=\bar\nu}
      q^{-(\mu_1,\bar\nu_1) + (\bar\mu_1,\bar\nu_1)
            -(\mu_2,\bar\nu_2) + (\bar\mu_2,\bar\nu_2)}
         q^{-(\mu_2,\alpha+\sigma)}
          f(\bar\mu_1,\bar\nu_1) f(\bar\mu_2,\bar\nu_2) \\
      &=C\cdot\sum_{\bar\nu_1+\bar\nu_2=\bar\nu}
         q^{-(\mu_1,\bar\nu_1) + (\bar\mu_1,\bar\nu_1)
            + (\bar\mu_2,\bar\nu_2)
            -(\mu_2,\bar\nu_1+\bar\nu_2+\alpha+\sigma) 
            + (\mu_2,\bar\nu_1)}
          f(\bar\mu_1,\bar\nu_1) f(\bar\mu_2,\bar\nu_2) \\
      &=C\cdot q^{-(\mu_2,\nu)}
         \sum_{\bar\nu_1+\bar\nu_2=\bar\nu}
         q^{(\mu_2-\mu_1,\bar\nu_1)} 
       q^{(\bar\mu_1,\bar\nu_1)} f(\bar\mu_1,\bar\nu_1)
       q^{(\bar\mu_2,\bar\nu_2)} f(\bar\mu_2,\bar\nu_2),
    \end{align*}
     Comparing this with the right hand side of the first equation of
     \eqref{f02} gives
    \begin{multline*}
         C\cdot q^{-(\mu_2,\nu)}
         \sum_{\bar\nu_1+\bar\nu_2=\bar\nu}
         q^{(\mu_2-\mu_1,\bar\nu_1)} 
       q^{(\bar\mu_1,\bar\nu_1)} f(\bar\mu_1,\bar\nu_1)
       q^{(\bar\mu_2,\bar\nu_2)} f(\bar\mu_2,\bar\nu_2) \\
       = \delta_{\mu_1,\mu_2} q^{-(\mu_2,\nu)+(\bar\mu_2,\bar\nu)}
       f(\bar\mu_2,\bar\nu),
    \end{multline*}
    and with the definition of $g(\bar\mu,\bar\nu)=|\Lambda_R/\Lambda'|
    q^{(\mu,\nu)} f(\mu,\nu)$ we get the  following equation
    \begin{equation*}
        \sum_{\bar\nu_1+\bar\nu_2 = \bar\nu} \delta_{(\mu_2-\mu_1\in
        \Cent^q(\Lambda_R))} q^{(\mu_2-\mu_1,\bar\nu_1)} g(\bar\mu_1,\bar\nu_1)
        g(\bar\mu_2,\bar\nu_2)
        = \delta_{\mu_1,\mu_2} g(\bar\mu_1,\bar\nu).
    \end{equation*}
    Analogously, we get the equation of the sum
    $\sum_{\bar\mu_1+\bar\mu_2=\bar\mu}$.

    We now consider the equations \eqref{f03}. Again, $\nu=\bar\nu+\alpha$
    as above.
    \begin{align*}
      \sum_{\nu\in\Lambda/\Lambda'} f(\mu,\nu) 
      &= \sum_{\nu} q^{-(\mu,\nu)+(\bar\mu,\bar\nu)} f(\bar\mu,\bar\nu) \\
      &= \sum_{\bar\nu} q^{(\bar\mu,\bar\nu)} f(\bar\mu,\bar\nu) 
         \sum_{\alpha\in\Lambda_R/\Lambda'} q^{-(\mu,\bar\nu+\alpha)} \\
      &= \sum_{\bar\nu} q^{-(\mu-\bar\mu,\bar\nu)} f(\bar\mu,\bar\nu) 
         \sum_{\alpha\in\Lambda_R/\Lambda'} q^{-(\mu,\alpha)} \\
      &= \delta_{(\mu\in \mathrm{Cent}^{[\ell]}(\Lambda_R))} |\Lambda_R/\Lambda'|
         \sum_{\bar\nu} q^{-(\mu-\bar\mu,\bar\nu)} f(\bar\mu,\bar\nu) \\
      &= \delta_{(\mu\in \mathrm{Cent}^{[\ell]}(\Lambda_R))} 
         \sum_{\bar\nu} q^{-(\mu,\bar\nu)} g(\bar\mu,\bar\nu) \\
      &= \delta_{\mu,0}. \qedhere
    \end{align*}
    \end{enumerate}
  \end{proof}

\section{The first type of equations} \label{group-equations}

  \subsection{Equations of \emph{group-type}}\label{sec:Fourier}
  \begin{defi}\label{GroupEquations}
    For an abelian group $G$ we define a set of equations for $|G|^2$
    variables $g(x,y)$, $x,y\in G$, which we call \emph{group-equations}.
    \begin{align}
      g(x,y) &= \sum_{y_1+y_2=y} g(x,y_1)g(x,y_2),\label{grpeq01}\\
      g(x,y) &= \sum_{x_1+x_2=x} g(x_1,y)g(x_2,y),\label{grpeq02}\\
      1 &= \sum_{y\in G} g(0,y),\label{grpeq03}\\
      1 &= \sum_{x\in G} g(x,0).\label{grpeq04}
    \end{align}
    Thus, there are $2|G|^2+2$ group-equations in $|G|^2$ variables with
    values in $\C$.
  \end{defi}

  These equations are the equations in Lemma \ref{g-eq} and the following
  Definition for central weight $\zeta=0$.

  \begin{thm}\label{solutionsgrpeq}
    Let $G$ be an abelian group of order $N$, $H_1,H_2$ subgroups with
    $|H_1|=|H_2|=d$.
    Let $\omega\colon H_1\times H_2\to\C^{\times}$ be a pairing of groups.
    Here, the group $G$ is written additively and $\C^{\times}$
    multiplicatively, thus we have $\omega(x,y)^d = 1$ for all $x\in H_1,
    y\in H_2$. Then the function
    \begin{equation}\label{g-solution}
      g\colon G\times G\to \C, ~ (x,y) \mapsto \frac 1d 
      \, \omega(x,y)\delta_{(x\in H_1)}\delta_{(y\in H_2)}
    \end{equation}
    is a solution of the group-equations \eqref{grpeq01}-\eqref{grpeq04} of
    $G$.
  \end{thm}

  \begin{proof}
    Let $G,H_1,H_2$ and $\omega$ be as in the theorem. We insert the
    function $g$ as in \eqref{g-solution} in the group-equation
    \eqref{grpeq01} of $G$. Let $x,y\in G$.
    \begin{align*}
      \sum_{y_1+y_2=y}g(x,y_1)g(x,y_2)
      &=\left(\frac 1d \right)^2\sum_{y_1+y_2=y} 
        \omega(x,y_1)\omega(x,y_2) 
        \delta_{(x\in H_1)}\delta_{(y_1\in H_2)}\delta_{(y_2\in H_2)} \\
      &=\left(\frac 1d\right)^2 \sum_{y_1+y_2=y}
        \omega(x,y_1+y_2)
        \delta_{(x\in H_1)}\delta_{(y_1\in H_2)}\delta_{(y_2\in H_2)} \\
      &=\left(\frac 1d\right)^2 |H_2|\, \omega(x,y) \delta_{(x\in
      H_1)}\delta_{(y\in H_2)}  \\
      &= g(x,y).
    \end{align*}
    Analogously for the sum in \eqref{grpeq02}. We now insert the function
    $g$ in \eqref{grpeq03}:
    \begin{equation*}
      \sum_{y\in G} g(0,y) 
      = \frac 1d\sum_{y\in G} \omega(0,y)\delta_{(y\in H_2)}
      = \frac 1d \sum_{y\in H_2}1
      = 1. \qedhere
    \end{equation*}
  \end{proof}

  \begin{que}
    Are these all solutions of the group-equations for a
    given group $G$? 
  \end{que}

  \subsection{Results for all fundamental groups of Lie algebras}
    We now treat the cases $G=\Z_N$ for $N\geq 1$ and $G=\Z_2\times\Z_2$,
    since these are the only examples of fundamental groups $\pi_1$ of root
    systems.

  \begin{thm}\label{allsolutionsgrpeq}
    In the following cases, the functions $g$ of Theorem
    \ref{solutionsgrpeq} are the only solutions of the group-equations
    \eqref{grpeq01}-\eqref{grpeq04} of $G$. 
    \begin{enumerate}[(a)]
      \item For $G=\Z_N$, the cyclic groups of order $N$. Here, we get
        $\sum_{d\mid N} d$ different solutions. 
      \item For $G=\Z_2\times\Z_2$. Here we get $35$ different solutions.
    \end{enumerate}
  \end{thm}

  \begin{proof}
    \begin{enumerate}[(a)]
      \item This is the content of \cite{LN14}, Theorem 5.6.
      \item We have checked this explicitly via {\sf MAPLE}. \qedhere
    \end{enumerate}
  \end{proof}

  \begin{expl}\label{cyclicSolutions}
    Let $G=\Z_N$, $N\geq 1$. For any divisor $d$ of $N$ there is a unique
    subgroup $H=\frac Nd\Z_N\cong \Z_d$ of $G$ of order $d$. By Theorem
    \ref{solutionsgrpeq} we have, that for any pairing $\omega\colon H\times
    H\to \C^{\times}$, the function $g$ as in \eqref{g-solution}
    is a solution of the group-equations \eqref{grpeq01}-\eqref{grpeq04}.
    We give the solution explicitly. For $H=\car{h}$, $h\in \frac Nd\Z_n$, we
    get a pairing $\omega\colon H\times H\to \C^{\times}$ by
    $\omega(h,h)=\xi$ with $\xi$ a $d$-th root of unity, not necessarily
    primitive. Thus, the function \eqref{g-solution} translates to
    \begin{equation}\label{eq:cyclicSolutions}
      g\colon G\times G\to\C,~ (x,y)\mapsto \frac 1d \, \xi^{\frac{xy}{(N/d)^2}}
      \delta_{(\frac Nd\mid x)}\delta_{(\frac Nd\mid y)}.
    \end{equation}
  \end{expl}

  \begin{expl}\label{2x2Solutions}
    Let $G=\Z_2\times\Z_2=\car{a,b}$. For $H_1=H_2=G$ there are $2^4=16$ possible
    parings, since a pairing is given by determining the values of
    $\omega(x,y)=\pm1$ for $x,y\in\{a,b\}$. In $G$, there are $3$ different
    subgroups of order $2$, hence there are $9$ possible pairs $(H_1,
    H_2)$ of groups $H_i$ of order $2$. For each pair, there are two
    possible choices for $\omega(x,y)=\pm1$, $x,y$ being the generators of
    $H_1$, resp. $H_2$. Thus, we get $18$ pairings for subgroups of order
    $d=2$.  For $H_1=H_2=\{0\}$ there is only one pairing, mapping $(0,0)$
    to $1$.  Thus, we have $35$ pairings in total.

    \begin{narrow}{-.6cm}{0cm}
    \begin{tabular}{c|c|c|c|c}
      $\#$ & $H_i\cong$ & $H_1$ & $H_2$ & $\omega$ \\ \hline
      $16$ & $\Z_2\times \Z_2$ & $\car{a,b}$ & $\car{a,b}$ & $\omega(x,y)=\pm1$
      for $x,y\in\{a,b\}$ \\ \hline
      $9\times 2$ & $\Z_2$ & $\car{x}$, $x\in\{a,b,a+b\}$ & $\car{y}$,
      $y\in\{a,b,a+b\}$ & $\omega(x,y)=\pm1$\\ \hline
      $1$ & $\Z_1$ & $\{0\}$ & $\{0\}$ & $\omega(0,0)=1$
    \end{tabular}
    \end{narrow}
  \end{expl}

\section{Quotient diamonds and the second type of equations}
  \label{diamond-equations}

  \subsection{Quotient diamonds and equations of \emph{diamond-type}}
  \begin{defi}\label{def:diamond}
    Let $G$ and $A$ be abelian groups and $B,C,D$ subgroups of $A$,
    such that $D=B\cap C$. We call a tuple
    $(G,A,B,C,D,\varphi_1,\varphi_2)$ with injective group morphisms
    $\varphi_1\colon A/B\rightarrow G^*=\Hom(G,\C^{\times})$ and
    $\varphi_2\colon A/C\rightarrow G$ a \emph{diamond for $G$}. We will
    visualize the situation with the following diagram
    \begin{equation*}
    \begin{tikzpicture}[scale=.5]
      \node (A) at (0,2) {$A$};
      \node (B) at (-2,0) {$B$};
      \node (C) at (2,0) {$C$};
      \node (D) at (0,-2) {$D$};
      \draw (A) -- (B) -- (D) -- (C) -- (A);
      \node (phi2) at (3.5,1) {$\hookrightarrow G$};
      \node (phi1) at (-3.5,1) {$G^*\hookleftarrow$};
    \end{tikzpicture}
    \end{equation*}
  \end{defi}

  \begin{defi}
    Let $(G,A,B,C,D,\varphi_1,\varphi_2)$ be a diamond for $G$.  For $a\in A$
    and not in $B\cap C$ we define the following equations for the $|G|^2$
    variables $g(x,y)$, $x,y\in G$:
    \begin{align}
      0 &= \sum_{y_1+y_2=y,y_i\in G} \varphi_1(a)(y_1)
      g(x,y_1)g(x+\varphi_2(a),y_2),\label{diaeq01}\\
      0 &= \sum_{x_1+x_2=x,x_i\in G} \varphi_1(a)(x_1)
      g(x_1,y)g(x_2,y+\varphi_2(a)),\label{diaeq02}\\
      0 &= \sum_{y\in G} \left(\varphi_1(a)(y)\right)^{-1}
      g(\varphi_2(a),y),\label{diaeq03}\\
      0 &= \sum_{x\in G} \left(\varphi_1(a)(x)\right)^{-1}
      g(x,\varphi_2(a))\label{diaeq04}.
    \end{align}
    We call this set of equations \emph{diamond-equations} for the diamond of
    $G$.
    Here, $\varphi_i(a)$ denotes the image of $a+B$, resp. $a+C$, for 
    $a\in A$ under $\varphi_1$, resp. $\varphi_2$. 
    
    These are up to $(|A|-1)(2|G|^2+2)$ equations in $|G|^2$ variables
    with values in $\C$.
  \end{defi}

  We show how these equations arise in the situation of Lemma \ref{g-eq}.

  \begin{lem}\label{diamonds}
    Let $G=\pi_1$, the fundamental group of a root system $\Phi$. Assume
    $\Lambda'$ is a sublattice of $\Lambda_R$, contained in
    $\Cent^q(\Lambda_W)$.  Let $A=\Cent^q(\Lambda_R)/\Lambda'$,
    $B=\Cent^q(\Lambda_W)/\Lambda'$, $C=\Cent^q(\Lambda_R)\cap \Lambda_R/\Lambda'$
    and $D=\Cent^q(\Lambda_W)\cap\Lambda_R/\Lambda'$.  Then there exist
    injections $\varphi_1\colon A/B\to\pi_1^*$ and $\varphi_2\colon
    A/C\to\pi_1$, such that $(G,A,B,C,D,\varphi_1,\varphi_2)$ is a diamond for
    $G$.
    \begin{equation*}
    \begin{tikzpicture}[scale=.7]
      \node (A) at (0,2) {$\mathrm{Cent}^{[\ell]}(\Lambda_R)/\Lambda'$};
      \node (B) at (-2,0) {$\mathrm{Cent}^{[\ell]}(\Lambda_W)/\Lambda'$};
      \node (C) at (2,0) {$\mathrm{Cent}^{[\ell]}(\Lambda_R)\cap\Lambda_R/\Lambda'$};
      \node (D) at (0,-2) {$\mathrm{Cent}^{[\ell]}(\Lambda_W)\cap\Lambda_R/\Lambda'$};
      \draw (A) -- (B) -- (D) -- (C) -- (A);
      \node (phi2) at (2.5,1) {$\hookrightarrow \pi_1$};
      \node (phi1) at (-2.5,1) {$\pi_1^*\hookleftarrow$};
    \end{tikzpicture}
    \end{equation*}
  \end{lem}

  \begin{proof}
    Recall from Lemmas \ref{CentR} and \ref{CentW}, that we have
    $\Cent^q(\Lambda_R)=\Lambda_W^{[\ell]}$ and
    $\Cent^q(\Lambda_W)\cap\Lambda_R=\Lambda_R^{[\ell]}$. We have $A/C\cong
    \Lambda_W^{[\ell]}/(\Lambda_W^{[\ell]}\cap\Lambda_R)$ and
    $\Lambda_W^{[\ell]}\subset\Lambda_W$,

    To show the existence of an injective morphism $\varphi_2:A/C\to\pi_1$, we
    define $\tilde\varphi_2$ on $\Lambda_W^{[\ell]}$ and calculate the kernel.
    By Definition \ref{elllattice}, the generators of $\Lambda_W^{[\ell]}$ are
    $\ell_{[i]}\lambda_i$ for all $i\in I$, with $\ell_{[i]}\df
    \ell/\gcd(\ell,d_i)$. Thus
    \begin{align*}
      \tilde\varphi_2 &\colon \Lambda_W^{[\ell]}\to \pi_1,~
      \ell_{[i]}\lambda_i \mapsto \ell_{[i]}\lambda_i +\Lambda_R
    \end{align*}
    gives a group morphism. Since
    $\Lambda'\subset\Lambda_R\cap\Lambda_W^{[\ell]}=\ker \tilde\varphi_2$,
    this induces a well-definend map $\varphi_2\colon A/\Lambda'\to\pi_1$.
    Obviously, the kernel of this map is $\Lambda_W^{[\ell]}\cap\Lambda_R$,
    hence the desired injection $\varphi_2\colon A/C\to\pi_1$ exists and is
    given by taking $\lambda+(\Lambda_W^{[\ell]}\cap\Lambda_R)$ modulo
    $\Lambda_R$, $\lambda\in \Lambda_W^{[\ell]}$.

    Now, we show the existence of $\varphi_1$. The map
    \begin{equation*}
      f\colon \Cent^q(\Lambda_R)\to\Hom(\Lambda_W,\C^{\times}),~
      \lambda\mapsto (\Lambda_W\to\C^{\times},~\eta\mapsto q^{(\lambda,\eta)})
    \end{equation*}
    is a group morphism. We define $g\colon \Hom(\Lambda_W,\C^{\times}) \to
    \Hom(\Lambda_W/\Lambda_R,\C^{\times})$ by $g(\psi)\df \psi\circ p$, where
    $p$ is the natural projection $\Lambda_W\to\Lambda_W/\Lambda_R$. Thus, the
    upper right triangle of the following diagram commutes.

    \begin{equation*}
    \begin{tikzcd}
      \Cent^q(\Lambda_R) \arrow{rd}[swap]{g\circ f} \arrow{r}{f} \arrow{d}[swap]{}
      & \Hom(\Lambda_W,\C^{\times}) \arrow{d}{g} \\
      \Cent^q(\Lambda_R)/\Cent^q(\Lambda_W) \arrow{r}[swap]{\varphi_1}
      & \Hom(\pi_1,\C^{\times}) \\
    \end{tikzcd}
    \end{equation*}
    There exists $\lambda\in\ker g\circ f$, iff $q^{(\lambda,\bar\eta)}=1$ for all
    $\bar\eta\in\pi_1$. Since $\lambda\in\Cent^q(\Lambda_R)$, this is equivalent
    to $q^{(\lambda,\eta)}=1$ for all $\eta\in\Lambda_W$, hence
    $\lambda\in\Cent^q(\Lambda_W)$. Thus, we get $\varphi_1$ as desired, which
    is well defined as map from
    $\Cent^q(\Lambda_R)/\Lambda'\big/\Cent^q(\Lambda_W)/\Lambda'$ since
    $\Lambda'\subset\Cent^q(\Lambda_W)=\ker f$. 
  \end{proof}

  \begin{lem}\label{necessCrit}
    Let $(G,A,B,C,D,\varphi_1,\varphi_2)$ be a diamond as in Lemma
    \ref{diamonds}. If $\Cent^q(\Lambda_W)\cap\Lambda_R/\Lambda'\neq
    0$, then none of the solutions of the group-equations
    \eqref{grpeq01}-\eqref{grpeq04} are solutions to the diamond-equations
    \eqref{diaeq01}-\eqref{diaeq04}.  Hence under our assumptions
    \ref{ass:Lambda'}, the existence of an $R$-matrix requires necessarily the
    choice $\Lambda'=\Cent^q(\Lambda_W)\cap\Lambda_R$.
  \end{lem}

  \begin{proof}
    If $\Cent^q(\Lambda_W)\cap\Lambda_R/\Lambda'\neq 0$, then there exist a root
    $\zeta\in\Cent^q(\Lambda_W)$, not contained in the kernel $\Lambda'$. Thus,
    there are diamond-equations with $\varphi_1(\zeta)=\one$ and
    $\varphi_2(\zeta)=0$, i.e. the set of equations:
    \begin{align}
      0 &= \sum_{y_1+y_2=y} g(x,y_1)g(x,y_2),\\
      0 &= \sum_{x_1+x_2=x} g(x_1,y)g(x_2,y),\\
      0 &= \sum_{y\in G} g(0,y),\\
      0 &= \sum_{x\in G} g(x,0).
    \end{align}
    Since this are group-equations as in Definition \ref{GroupEquations}, but
    with left-hand side equal to $0$, solutions of the group-equations
    does not solve the diamond-equations in this situation.
  \end{proof}

  Before examining in which case a solution of the group-equations as in Theorem
  \ref{solutionsgrpeq} is also a solution of the diamond-equations
  \eqref{diaeq01}-\eqref{diaeq04}, we show that it is sufficient to check the
  diamond-equations \eqref{diaeq03} and \eqref{diaeq04}. 

  \begin{lem}\label{lem:onlyScalingEq}
    Let $G$ be an abelian group of order $N$, $H_1$, $H_2$ subgroups with
    $|H_1|=|H_2|=d$ and $\omega\colon H_1\times H_2\to \C^{\times}$ a
    group-pairing, such that $g\colon G\times G\to \C,~ (x,y)\mapsto
    1/d~\omega(x,y) \delta_{(x\in H_1)}\delta_{(y\in H_2)}$ is a solution of
    the group-equations \eqref{grpeq01}-\eqref{grpeq04}, as in Theorem
    \ref{solutionsgrpeq}. Then the following holds: \\
    If $g$ is a solution of the diamond-equations \eqref{diaeq01},
    \eqref{diaeq02}, then $g$ solves the diamond-equations \eqref{diaeq03},
    \eqref{diaeq04} as well.
  \end{lem}

  \begin{proof}
    Let $g$ be a solution of the group-equations as in Theorem
    \ref{solutionsgrpeq}. Assume that $g$ solves \eqref{diaeq01} and
    \eqref{diaeq02}. Let $\varphi_1$, $\varphi_2$ as in Definition
    \ref{def:diamond} and $0\neq \zeta\in A$ a non-trivial central weight.
    Then, for $x,y\in G$ we get by inserting $g$ in \eqref{diaeq01}
    \begin{align*}
      0 &= \sum_{y_1+y_2=y} \varphi_1(\zeta)(y_1) g(x,y_1)
           g(x+\varphi_2(\zeta),y_2) \\
           &= \sum_{y_1+y_2=y} \varphi_1(\zeta)(y_1) \frac1{d^2} 
           \omega(x,y_1) \omega(x+\varphi_2(\zeta),y_2) 
           \delta_{(x\in H_1)}\delta_{(y_1\in H_2)}
           \delta_{(x+\varphi_2(\zeta)\in H_1)}\delta_{(y_2\in H_2)} \\
           &= 
           \delta_{(x\in H_1)}\delta_{(y\in H_2)}
           \delta_{(\varphi_2(\zeta)\in H_1)}
            \frac1{d^2}  
            \sum_{\substack{y_1+y_2=y \\ y_1,y_2\in H_2}} 
            \varphi_1(\zeta)(y_1)
            \omega(x,y_1) \omega(x,y_2) \omega(\varphi_2(\zeta),y_2) \\
           &=\delta_{(x\in H_1)}\delta_{(y\in H_2)}
           \delta_{(\varphi_2(\zeta)\in H_1)}
            \frac1{d^2} \omega(x,y) 
            \sum_{\substack{y_1+y_2=y \\ y_1,y_2\in H_2}} 
            \varphi_1(\zeta)(y_1) \omega(\varphi_2(\zeta),y_2) \\
           &=\delta_{(x\in H_1)}\delta_{(y\in H_2)}
           \delta_{(\varphi_2(\zeta)\in H_1)}
            \frac1{d^2} \omega(x,y) 
           \sum_{y_2\in H_2} \varphi_1(\zeta)(y-y_2) 
           \omega(\varphi_2(\zeta),y_2) \\
           &= \delta_{(x\in H_1)}\delta_{(y\in H_2)} 
           \delta_{(\varphi_2(\zeta)\in H_1)} \frac1{d^2}  
              \omega(x,y) \varphi_1(\zeta)(y) 
              \sum_{y_2\in H_2} \varphi_1(\zeta)(y_2)^{-1} 
              \omega(\varphi_2(\zeta),y_2).
    \end{align*}
    In particular, this holds for $x=y=0$, and in this case the expression
    vanishes iff
    \begin{equation*}
      \delta_{(\varphi_2(\zeta)\in H_1)} \frac1{d^2}    
              \sum_{y\in H_2} \varphi_1(\zeta)(y)^{-1} 
              \omega(\varphi_2(\zeta),y)=0,
    \end{equation*}
    which is \eqref{diaeq03}. Analogously, it follows that if $g$ solves
    \eqref{diaeq02} it solves \eqref{diaeq04}.
  \end{proof}
                
  \subsection{Cyclic fundamental group $G=\Z_N$}
    In the following, $G$ will always be a fundamental group of a simple
    complex Lie algebra, hence either cyclic or equal to $\Z_2\times\Z_2$ for
    the case $D_{n}$, $n$ even. In this section, we will derive some results
    for the cyclic case.

    In Example \ref{cyclicSolutions} we have given solutions of the
    group-equations for $G=\Z_{N}$, i.e. for all $d\mid N$ the functions
    \begin{equation}
      g\colon G\times G\to \C,~ (x,y)\mapsto \frac 1d \xi^{\frac{xy}{(N/d)^2}}
        \,\delta_{(\frac Nd\mid x)}\delta_{(\frac Nd\mid y)}
    \end{equation}
    with $\xi$ a $d$-th root of unity, not necessarily primitive. In the
    following, we denote by $\xi_d$ the primitive $d$-th root of
    unity $\exp(2\pi i/d)$.

    \begin{lem}\label{solutionsDiamondsCyclic}
      Let $l\geq 2$, $m\in\N$ and $G=\car{\lambda}\cong\Z_N$. We consider the
      following diamonds
      $(G,A,B,C,D,\varphi_1,\varphi_2)$ 
      with $A=\car{a}\cong\Z_N$
      and injections $\varphi_1$ and $\varphi_2$ given by
      \begin{alignat*}{3}
        \tilde\varphi_1&\colon A\to G^*,&~ a&\mapsto (\xi_N^{m})^{(-)}, 
        \quad\text{with }
        (\xi_N^m)^{(-)}\colon G\to\C^{\times},~ x\mapsto \xi_N^{mx}, \\
        \tilde\varphi_2 &\colon A\to G,&~ a&\mapsto l \lambda,
      \end{alignat*}
      with primitive $N$-th root of unity $\xi_N$, $B=\ker\tilde\varphi_1$ and
      $C=\ker\tilde\varphi_2$ and $D=\{0\}$.

      Possible solutions of the group-equations
      \eqref{grpeq01}-\eqref{grpeq04} are given for any choice of integers
      $1\leq k\leq d$ and $d\mid N$ as in Example \ref{cyclicSolutions} by
      \begin{equation}\label{gCyclicSolution}
        g\colon G\times G\to \C,~ (x,y)\mapsto \frac 1d
        \left(\xi_d^k\right)^{\frac{xy}{(N/d)^2}}
          \,\delta_{(\frac Nd\mid x)}\delta_{(\frac Nd\mid y)},
      \end{equation}
      with primitive $d$-th root of unity $\xi_d=\exp(2\pi i/d)$. These are
      solutions also to the diamond-equations \eqref{diaeq01}-\eqref{diaeq04},
      iff $N\mid m,l$ or the following condition hold:
      \begin{equation}
        \gcd(N,dl,kl-\frac Nd m) = 1.
      \end{equation}
    \end{lem}

    \begin{proof}
      For $N\mid m,l$ there is no non-trivial diamond-equation, hence all
      solutions of the group-equations as in Example \ref{cyclicSolutions} are
      possible. Assume now, that not both $N\mid m$ and $N\mid l$.  We insert
      the function $g$ from \eqref{gCyclicSolution} in the diamond-equations
      \eqref{diaeq01}-\eqref{diaeq04} and get requirements for $d,k,l$ and
      $N$. By Lemma \ref{lem:onlyScalingEq} it is sufficient to consider only
      equations \eqref{diaeq03} and \eqref{diaeq04}. Since for cyclic $G$ the
      function $g$ is symmetric we choose equation \eqref{diaeq03} for the
      calculation. In the following we omit the $\tilde{}$ on the maps
      $\tilde\varphi_{1/2}\colon A\to G^*$, resp. $G$.  Let $1\leq z<N$, $a\in
      A$ and $y\in G$, then
      \begin{align*}
        \sum_{y=1}^N 
        \left(\varphi_1(za)(y)\right)^{-1}
        g(\varphi_2(za), y) 
        &= \frac 1d \sum_{y=1}^N \xi_N^{-zm y}
        \left(\xi_d^k\right)^{\frac{zl y}{(N/d)^2}}
        \, \delta_{(\frac Nd \mid zl)}\delta_{(\frac Nd \mid y)} \\
        &= \frac 1d \sum_{y=1}^N \xi_d^{-\frac{z my}{N/d}} \left(
        \xi_d^{\frac{zkl}{(N/d)}} \right)^{\frac{y}{N/d}}
        \, \delta_{(\frac Nd \mid zl)}\delta_{(\frac Nd \mid y)}, \\
        &= \frac 1d \sum_{y'=1}^d \left(
        \xi_d^{-zm+\frac{zkl}{(N/d)}} \right)^{y'}
        \, \delta_{(\frac Nd \mid zl)},
      \end{align*}
      with the substitution $y'=y/(N/d)$. This sum equals $0$ iff $N/d\nmid
      zl$ or $d\nmid z(kl/(N/d)-m)$. This is equivalent to $N\nmid zdl$
      or $N\nmid z(kl-(N/d)m)$, hence $N\nmid \gcd(zdl,z(kl-(N/d)m))$. Since
      this condition has to be fulfilled for all $z$ we get that $N\nmid
      \gcd(dl, kl-(N/d) m)$, hence $\gcd(N,ld,kl-(N/d)m)=1$.
    \end{proof}

    We spell out the condition for explicit values $m$ and $l$.
    \begin{expl}\label{expl:DiamondSolutionsCyclci}
      Let $G=\Z_N=\car{\lambda}$, $l\geq 2$, $m\in\N$ and diamond
      $(G,A,B,C,D,\varphi_1,\varphi_2)$ as in Lemma
      \ref{solutionsDiamondsCyclic}. Depending on $m,l$ we get the following
      criteria for solutions of the diamond-equations. Here, we give
    $\varphi_1$ and $\varphi_2$ shortly by the generator of its image.
      \begin{enumerate}[(I)]
        \item If $N\mid m$ and $N\mid l$ we have the diamond
          $(\Z_N,\Z_N,\Z_N,\Z_N,\Z_1,$ $1,0)$ and all solutions of the form
          \eqref{gCyclicSolution} are also solutions to the diamond-equations
          \eqref{diaeq01}-\eqref{diaeq04}. (Since $B,C= A$, there are no
          non-trivial diamond-equations.)
        \item If $N\mid m$ and $N\nmid l$ we have the
          diamond
          $(\Z_N,\Z_N,\Z_N,\Z_{\gcd(l,N)},\Z_1,1,l\lambda)$.  In this case the
          function $g$ as in \eqref{gCyclicSolution} is a solution to the
          diamond-equations \eqref{diaeq01}-\eqref{diaeq04} if
          $\gcd(N,dl,kl)=1$.
        \item If $\gcd(m,N)=1$ and $N\nmid l$ we have the
          diamond
          $(\Z_N,\Z_N,\Z_1,\Z_{\gcd(l,N)},\Z_1,\xi_N,$
          $l\lambda)$.  In this case the function $g$ as in
          \eqref{gCyclicSolution} is a solution to the diamond-equations
          \eqref{diaeq01}-\eqref{diaeq04} if
          \begin{equation}\label{eq:gcd-condi}
            \gcd(N,dl,kl-\frac Nd m)=1.
          \end{equation}
          In most cases, $N$ is prime or equals $1$, hence we consider the
          two special cases
          \begin{enumerate}[(1)]
            \item If $d=1$, \eqref{eq:gcd-condi} simplifies to
              $\gcd(N,l,l-Nm)=1$, which is equivalent
              to $\gcd(N,l)=1$.
            \item If $d=N$, \eqref{eq:gcd-condi} simplifies to
              $\gcd(N,lN,kl-m)=1$, which is equivalent
              to $\gcd(N,kl-m)=1$.
          \end{enumerate}
      \end{enumerate}
    \end{expl}

    Finally, we consider the Lie algebras with cyclic fundamental group in
    question and determine the values $m$ and $l$ according to the Lie
    theoretic data and thereby the corresponding diamonds.

    \begin{expl}\label{expl:diamonds}
      Let $G=\Z_N$ be the fundamental group of a simple complex Lie algebra $\g$,
      generated by the fundamental dominant weight $\lambda_n$. Let
      $\ell\in\N$, $\ell>2$, $q=\exp(2\pi i/\ell)$, $\ell_{[n]}=
      \ell/\gcd(\ell,d_n)$, $m_{[n]}\df N(\lambda_n,\lambda_n)/\gcd(\ell,d_n)$
      and $(G,A,B,C,D,\varphi_1,\varphi_2)$ be a diamond as in Lemma
      \ref{diamonds}, such that the corresponding diamond-equations
      \eqref{diaeq01}-\eqref{diaeq04} have a solution that is also a solution
      to the group-equations \eqref{grpeq01}-\eqref{grpeq04}. Then, the
      diamond is 
       \begin{equation}\label{eq:possibleDiamond}
         (G,\Z_N,\Z_{\gcd(m_{[n]},N)},\Z_{\gcd(\ell_{[n]},N)},
         \Z_1,\varphi_1,\varphi_2),
       \end{equation}
      with injections
      \begin{alignat*}{3}
        \varphi_1&\colon A\to G^*,&~ \ell_{[n]}\lambda_n&\mapsto
        (\xi_N^{m_{[n]}})^{(-)}, 
        \quad\text{with }
        (\xi_N^{m_{[n]}})^{(-)}\colon G\to\C^{\times},~ x\mapsto \xi_N^{m_{[n]}x},
        \\
        \varphi_2 &\colon A\to G,&~ \ell_{[n]}\lambda_n&\mapsto \ell_{[n]}
        \lambda_n,
      \end{alignat*}
      with primitive $N$-th root of unity $\xi_N=\exp(2\pi i/N)$. The group
      $A=\Cent^q(\Lambda_R)/\Lambda'=\Lambda_W^{[\ell]}/\Lambda_R^{[\ell]}$ is
      generated by $\ell_{[n]}\lambda_n$ and
      $q^{(\ell_{[n]}\lambda_n,\lambda_n)}=
      (\xi_N^N)^{(\lambda_n,\lambda_n)/\gcd(\ell,d_n)}$.
      Since the order of $\xi_N^{m_{[n]}}$ in $\C^{\times}$ is
      $N/\gcd(m_{[n]},N)$ and the order of  $\ell_{[n]}$ in $\Z_N$ is
      $N/\gcd(\ell_{[n]},N)$, the injections $\varphi_1,\varphi_2$ determine
      the diamond \eqref{eq:possibleDiamond}.
      
    In the following table, we give the values $\ell_{[n]}$ and $m_{[n]}$ for
    all root systems of simple Lie algebras with cyclic fundamental group.
    \begin{equation*}\label{mN}
  \renewcommand{\arraystretch}{1.1}
      \begin{tabular}{|c||c|c|c|c|c|c|c|c|c|c|c|}
        \hline
        \Multi{2}{$\g$} 
        &\Multi{2}{$A_{n\geq1}$} &\Multi{2}{$B_{n\geq2}$}
        &\multicolumn{2}{c|}{\Multi{2}{$C_{n\geq3}$}} &$D_{n\geq5}$
        &\Multi{2}{$E_6$} &\Multi{2}{$E_7$} &\Multi{2}{$E_8$}
        &\Multi{2}{$F_4$} &\multicolumn{2}{c|}{\Multi{2}{$G_2$}} \\
        & & &\multicolumn{2}{c|}{} &$n$ odd & & & & &\multicolumn{2}{c|}{}
        \\\hline
    $\pi_1$ & $\Z_{n+1}$ &$\Z_2$ &\multicolumn{2}{c|}{$\Z_2$} &$\Z_4$
        &$\Z_3$ &$\Z_2$ &$\Z_1$ &$\Z_1$ &\multicolumn{2}{c|}{$\Z_1$}
        \\\hline
    $N$ & $n+1$ &$2$ &\multicolumn{2}{c|}{$2$} &$4$ &$3$ &$2$ &$1$
        &$1$ &\multicolumn{2}{c|}{$1$} \\\hline \hline
    $d_n$ & $1$ & $1$ & \multicolumn{2}{c|}{$2$} & $1$ & $1$ & $1$ &
        $1$ & $1$ & \multicolumn{2}{c|}{$3$} \\\hline
    $\ell$ & all & all & $2\nmid\ell$ & $2\mid\ell$ & all & all & all & all
           & all & $3\nmid\ell$ & $3\mid\ell$ \\\hline
    $\gcd(\ell,d_n)$ & $1$ & $1$ & $1$ & $2$ & $1$ & $1$ & $1$ & $1$ &
        $1$ & $1$ & $3$  \\\hline
    $(\lambda_n,\lambda_n)$ & $\frac n{n+1}$ & $\frac n2$ & 
        \multicolumn{2}{c|}{$n$} & $\frac n4$ & $\frac 43$ & $\frac
        32$ & $2$ & $1$ & \multicolumn{2}{c|}{6}  \\\hline\hline
    $\ell_{[n]}$ & $\ell$ & $\ell$ & $\ell$ & $\ell/2$ & $\ell$ & $\ell$ 
        & $\ell$ & $\ell$ & $\ell$ & $\ell$ & $\ell/3$ \\\hline
    $m_{[n]}$ & $n$ & $n$ & $2n$ & $n$
        & $n$ & $4$ & $3$ & $2$ & $1$ & $6$ & $2$ \\ \hline\hline
    cases & (III) & (I)-(III) & (II) & (I)-(III) & (III) & (III) & (III) & (I)
        & (I) & (I) & (I) \\\hline
      \end{tabular}
    \end{equation*}
    In the last row we indicate which cases in Example
    \ref{expl:DiamondSolutionsCyclci} apply. This will guide the proof of
    Theorem \ref{Thm:solutions}. Note that case (II) only appears for $B_n$,
    $n$ even and $\ell$ odd, and for $C_n$, even $n$ and $\ell\equiv2\mod4$ or
    odd $\ell$.
    \end{expl}

  \subsection{Example: $B_2$}
    For $\g$ with root system $B_2$ we have $\pi_1=\Z_2$.  There is one long
    root, $\alpha_1$, and one short root, $\alpha_2$, hence $d_1=2$ and
    $d_2=1$. The symmetrized Cartan matrix $\tilde C$ is given below. The
    fundamental dominant weights $\lambda_1,\lambda_2$ are given as in \cite{Hum72},
    Section 13.2. Here, $\lambda_1$ is a root and
    $\lambda_2$ is the generator of the fundamental group $\Z_2$. The
    matrix $\id_W^R$ gives the coefficients of the fundamental dominant
    weights in the basis $\{\alpha_1,\alpha_2\}$.

    \begin{equation*}
      \tilde C=
      \begin{pmatrix}
        4 & -2  \\
        -2&  2 
      \end{pmatrix}
      \qquad
      \id_W^R= 
      \begin{pmatrix}
        1 & \frac12  \\
        1 & 1 
      \end{pmatrix}
    \end{equation*}


    Thus, $(\lambda_2,\lambda_2)=1$.
    The lattice diamonds, depending on $\ell$, are:
    \begin{enumerate}[(i)]
      \item For odd $\ell$ we have $A=\ell\Lambda_W$ and $C=D=\ell\Lambda_R$.
        Since $(\lambda_2,\lambda_2)=1$, we have $B=\ell\Lambda_W$. (Since
        $(\lambda_n,\lambda_n)=n/2$, in the general case $B_n$, the group
        $\Cent^q(\Lambda_W)$ depends on $n$: for
        even $n$ we have $B=\ell\Lambda_W$, and $B=\ell\Lambda_R$ for odd
        $n$.)
      \item For even $\ell$ we have $A=C=B= \ell\car{\frac 12\lambda_1,
        \lambda_2}$ and $D=
        \ell\car{\frac12\alpha_1,\alpha_2}$.
        (Again, $B$ depends on $n$, hence we have
        $B=\car{\frac12\lambda_1,\ldots,\frac12\lambda_{n-1},\lambda_n}$ if $n$
        is even and
        $B=\car{\frac12\lambda_1,\ldots,\frac12\lambda_{n-1},2\lambda_n}$ if $n$
        is odd.)
      \end{enumerate}

    We calculate the quotient diamonds for kernel
    $\Lambda'=\Lambda_R^{[\ell]}$ since by the necessary criterion of Lemma
    \ref{necessCrit}, this is the only case where possible
    solutions exist. We calculate Lusztig's kernel $2\Lambda_R^{(\ell)}$ as
    well and compare it with $\Lambda_R^{[\ell]}$. 
    We then determine the solutions of the corresponding diamond-equations
    according to Example \ref{expl:DiamondSolutionsCyclci}.
    \begin{enumerate}[(i)]
      \item For odd $\ell$ it is $\Lambda_R^{[\ell]}=\ell\Lambda_R \neq
        2\ell\Lambda_R = 2\Lambda_R^{(\ell)}$, $\ell_{[n]}=\ell$ and
        $m_{[n]}=n=2=N$. Thus, for
        $\Lambda'=\Lambda_R^{[\ell]}$ the quotient diamond is given by
        $(\Z_2,\Z_2,\Z_2,\Z_1,\Z_1,1,\lambda_2)$. By Example
        \ref{expl:DiamondSolutionsCyclci} (II), one has to check for which
        $d,k$ it is $\gcd(2,d\ell,k\ell)=1$.  This gives the 2 solutions:
        $(d,k)=(1,1)$ and $(d,k)=(2,1)$.
      \item[(ii.a)] For $\ell\equiv 2\mod 4$ it is $\Lambda_R^{[\ell]}=
        \ell\car{\frac12 \alpha_1, \ldots \frac12 \alpha_{n-1}, \alpha_n}
        \neq\ell\Lambda_R = 2\Lambda_R^{(\ell)}$, $\ell_{[n]}=\ell$ and
        $m_{[n]}=N$ as above. Here, we have $\gcd(\ell,N)=N=2$, thus for
        $\Lambda'=\Lambda_R^{[\ell]}$ we get the quotient diamond
        $(\Z_2,\Z_2,\Z_2,\Z_2,\Z_1,1,0)$. Thus, all 3 solutions of the
        group-equations are solutions to the diamond-equations as well by
        \ref{expl:DiamondSolutionsCyclci} (I).
    \item[(ii.b)] For $\ell\equiv 0\mod 4$ it is $\Lambda_R^{[\ell]}=
      \ell\car{\frac12 \alpha_1, \ldots \frac12 \alpha_{n-1}, \alpha_n}
      =2\Lambda_R^{(\ell)}$. Thus in this case the quotient diamond as in
      (ii.a) is the same for Lusztig's kernel, namely
      $(\Z_2,\Z_2,\Z_2,\Z_2,\Z_1,1,$ $0)$ and again all 3 solutions of the
      group-equations are solutions to the diamond-equations as well.
    \end{enumerate}
    \begin{equation*}
    \begin{tikzpicture}[scale=.4]
      \node (A) at (0,2) {$\Z_2$};
      \node (B) at (-2,0) {$\Z_2$};
      \node (C) at (2,0) {$\Z_1$};
      \node (D) at (0,-2) {$\Z_1$};
      \draw (A) -- (B) -- (D) -- (C) -- (A);
      \node (phi2) at (3.5,1) {$\mapsto \lambda_2$};
      \node (phi1) at (-3.5,1) {$1\mapsfrom$};
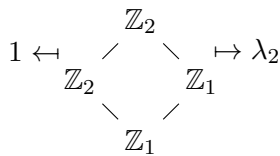
      \node (kom) at (0,-4) {quotient diamond in case (i)};
    \end{tikzpicture}
    \qquad\qquad
    \begin{tikzpicture}[scale=.4]
      \node (A) at (0,2) {$\Z_2$};
      \node (B) at (-2,0) {$\Z_2$};
      \node (C) at (2,0) {$\Z_2$};
      \node (D) at (0,-2) {$\Z_1$};
      \draw (A) -- (B) -- (D) -- (C) -- (A);
      \node (phi2) at (3.5,1) {$\mapsto 0$};
      \node (phi1) at (-3.5,1) {$1\mapsfrom$};
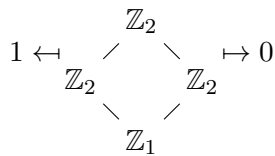
      \node (kom) at (0,-4) {quotient diamond in cases (ii.a), (ii.b)};
    \end{tikzpicture}
    \end{equation*}

\section{Proof of Theorem A}\label{solutions}
    We treat the root systems case by case and determine the solutions of diamond
    equations in Section \ref{diamond-equations} which are of the
    form
      \begin{equation*}
        g\colon G\times G\to \C,~ (x,y)\mapsto 
          \frac 1d\omega(x,y)\delta_{(x\in H_1)}\delta_{(y\in H_2)}
      \end{equation*}
      with subgroups $H_1,H_2$ of $G=\pi_1$ as in Theorem
      \ref{solutionsgrpeq}. \\
    For this, we first determine the
    lattices
    $A=\Cent^q(\Lambda_R)=\Lambda_W^{[\ell]},~ 
    B=\Cent^q(\Lambda_W),~ C=\Cent^q(\Lambda_R)\cap\Lambda_R,~
    D=\Cent^q(\Lambda_W)\cap\Lambda_R=\Lambda_R^{[\ell]}$, depending on $\ell$.
    For the Lie algebras with cyclic fundamental group (all but for root
    system $D_n$ with even $n$), we then determine the values $m_{[n]}$ and
    $\ell_{[n]}$, depending on $\ell$, $n$ and the order of $\pi_1$, and
    thereby the quotient diamonds and which solutions of the group equations
    are solutions to the corresponding diamond equations. In these cases, the
    $\omega$-part of the solutions to the group equations are of the form
      \begin{equation*}
        \omega\colon H\times H\to \C^{\times},~ (x,y)\mapsto 
        \left(\xi_d^k\right)^{\frac{xy}{(N/d)^2}}
      \end{equation*}
    for subgroup $H=\frac Nd \Z_{N}$ of $\pi_1$ of order $d$. We give
    the solutions by pairs $(d,k)$, which we determine
    by applying Lemma \ref{solutionsDiamondsCyclic} and Example
    \ref{expl:DiamondSolutionsCyclci}. An overview of the possible cases
    gives Example \ref{expl:diamonds}. \\
    For $D_n$ with even $n$ and fundamental group $\Z_2\times\Z_2$ we also
    determine all quotient diamonds (depending on $\ell$) and check which
    solutions of the group equations solve the diamond equations in a rather
    case by case calculation.

    \begin{enumerate}[(1)]

      \item For $\g$ with root system $A_n$, $n\geq 1$, we have
        $\pi_1=\Z_{n+1}$ for all $n$. The simple roots are
        $\alpha_1,\ldots,\alpha_{n}$ and $d_i=1$ for $1\leq i\leq n$. The
        symmetrized Cartan matrix $\tilde C$ is given below. The fundamental
        dominant weights $\lambda_i$ are given as in \cite{Hum72}, Section
        13.2, and $\lambda_n$ is the generator of the fundamental group
        $\Z_{n+1}$. The matrix $\id_W^R$ gives the coefficients of the
        fundamental dominant weights in the basis
        $\{\alpha_1,\ldots,\alpha_n\}$.

        \begin{gather*}
          \tilde C=
          \begin{pmatrix}
            2 & -1 & 0 & .& .& .& .&0& \\
            -1&  2 &-1 & 0& .& .& .&0& \\
            0 & -1 & 2 &-1& 0& .& .&0& \\
            . &  . & . & .& .& .& .&.& \\
            0 &  0 & 0 & 0& .& .&-1& 2
          \end{pmatrix} \\
          \id_W^R=a_{ij} \text{ with }
          a_{ij} = \begin{cases}
            \frac1{n+1} i(n-j+1), & \text{if } i\leq j,\\
            \frac1{n+1} j(n-i+1), & \text{if } i> j.
          \end{cases}
        \end{gather*}

        The lattice diamonds, depending on $\ell$, are:
        \begin{enumerate}[(i)]
          \item For even $\ell$ we have $A=\ell\Lambda_W$, $B=D=\ell\Lambda_R$
            and $C=\ell/\gcd(n+1,\ell)\Lambda_W$.
          \item For odd $\ell$: the same lattices as in (i).
        \end{enumerate}

        We calculate the quotient diamonds for kernel
        $\Lambda'=\Lambda_R^{[\ell]}$ and compare it with Lusztig's kernel
        $2\Lambda_R^{(\ell)}$. 
        We then determine the solutions of the corresponding diamond-equations
        according to Example \ref{expl:DiamondSolutionsCyclci}.
        \begin{enumerate}[(i)]
          \item For odd $\ell$ it is $\Lambda_R^{[\ell]}=\ell\Lambda_R \neq
            2\ell\Lambda_R = 2\Lambda_R^{(\ell)}$, $\ell_{[n]}=\ell$ and
            $m_{[n]}=n$. Thus, the quotient diamond is given by
            $(\Z_{n+1},\Z_{n+1},\Z_1,\Z_{\gcd(\ell,n+1)},\Z_1,\xi_{n+1},
            \ell\lambda_n)$, hence we are in case (III) of Example
            \ref{expl:DiamondSolutionsCyclci}. We get solutions $(d,k)$ iff
            $\gcd(n+1,d\ell,k\ell-\frac{n+1}d n)=1$. 
          \item For odd $\ell$ it is $\Lambda_R^{[\ell]}=\ell\Lambda_R =
            2\Lambda_R^{(\ell)}$, $\ell_{[n]}=\ell$ and
            $m_{[n]}=n$. Thus, the quotient diamonds and solutions are as in
            (i). 
        \end{enumerate}

      \item For $\g$ with root system $B_n$, $n\geq 2$, we have $\pi_1=\Z_2$
        for all $n$. The long simple roots are $\alpha_1,\ldots,\alpha_{n-1}$
        and the short simple root $\alpha_n$, hence $d_i=2$ for $1\leq i\leq
        n-1$ and $d_n=1$. The symmetrized Cartan matrix $\tilde C$ is given
        below. The fundamental dominant weights $\lambda_i$ are given as in
        \cite{Hum72}, Section 13.2. Here, $\lambda_1,\ldots,\lambda_{n-1}$ are
        roots and $\lambda_n$ is the generator of the fundamental group
        $\Z_2$. The matrix $\id_W^R$ gives the coefficients of the fundamental
        dominant weights in the basis $\{\alpha_1,\ldots,\alpha_n\}$.

        \begin{equation*}
          \tilde C=
          \begin{pmatrix}
            4 & -2 & 0& 0& .& & & 0  \\
            -2&  4 &-2& 0& .& & & 0  \\
             .&   .& .& .& .&.&.&.&  \\
            0 &  0 & 0&  & .&-2&4&-2 \\
            0 &  0 & 0&  & .& 0&-2&2
          \end{pmatrix}
          \quad
          \id_W^R= 
          \begin{pmatrix}
            1 & 1 & . & . & 1 & \frac 12 \\
            1 & 2 & . & . & 2 & 1 \\
            . & . & . & . & . & . \\
            1 & 2 & 3 & . & n-1 & \frac {n-1}2 \\
            1 & 2 & 3 & . & n-1 & \frac n2
          \end{pmatrix}
        \end{equation*}

        The lattice diamonds, depending on $\ell$, are:
        \begin{enumerate}[(i)]
          \item For odd $\ell$ we have $A=\ell\Lambda_W$ and
            $C=D=\ell\Lambda_R$. Since $(\lambda_n,\lambda_n)=n/2$, the group
            $\Cent^q(\Lambda_W)$ depends on $n$. It is $B=\ell\Lambda_W$ for
            even $n$ and $B=\ell\Lambda_R$ for odd $n$.
          \item For even $\ell$ we have $A=C=
            \ell\car{\frac12\lambda_1,\ldots,\frac12\lambda_{n-1}, \lambda_n}$
            and $D=\ell\car{\frac12 \alpha_1,\dots,$
            $\frac12\alpha_{n-1},\alpha_n}$.  Again, $B$ depends on $n$, and
            we have
            $B=\ell\car{\frac12\lambda_1,\dots,\frac12\lambda_{n-1},\lambda_n}$
            for even $n$ and
            $B=\ell\car{\frac12\lambda_1,\dots,\frac12\lambda_{n-1},2\lambda_n}$
            for odd $n$.
        \end{enumerate}

        We calculate the quotient diamonds for kernel
        $\Lambda'=\Lambda_R^{[\ell]}$ and compare it with Lusztig's kernel
        $2\Lambda_R^{(\ell)}$. 
        We then determine the solutions of the corresponding diamond-equations
        according to Example \ref{expl:DiamondSolutionsCyclci}.
        \begin{enumerate}[(i)]
          \item For odd $\ell$ it is $\Lambda_R^{[\ell]}=\ell\Lambda_R \neq
            2\ell\Lambda_R = 2\Lambda_R^{(\ell)}$, $\ell_{[n]}=\ell$ and
            $m_{[n]}=n$. In this case, the quotient diamond is given by either
            $(\Z_2,\Z_2,\Z_2,\Z_1,\Z_1,1,\lambda_n)$ for even $n$, or by
            $(\Z_2,\Z_2,\Z_1,\Z_1,\Z_1,-1,\lambda_n)$ for odd $n$. Thus we
            are either in case (II), or in case (III) of Example
            \ref{expl:DiamondSolutionsCyclci}. In the first case (even $n$) we
            get solutions by $(d,k)=(1,1)$ and $(2,1)$. For odd $n$ we get
            solutions $(d,k)=(1,1)$ and $(2,2)$. 
          \item[(ii.a)] For $\ell\equiv 2\mod 4$ it is $\Lambda_R^{[\ell]}=
            \ell\car{\frac12 \alpha_1, \ldots ,\frac12 \alpha_{n-1}, \alpha_n}
            \neq\ell\Lambda_R = 2\Lambda_R^{(\ell)}$, $\ell_{[n]}=\ell$ and
            $m_{[n]}=n$. The quotient diamond is given by either
            $(\Z_2,\Z_2,\Z_2,\Z_2,\Z_1,1,0)$ for even $n$, or by
            $(\Z_2,\Z_2,\Z_1,\Z_2,\Z_1,-1,0)$ for odd $n$. Thus we are in
            either in case (I) or in case (III) of Example
            \ref{expl:DiamondSolutionsCyclci}. In the first case (even $n$) we
            get all possible 3 solutions $(d,k)=(1,1)$, $(2,1)$ and $(2,1)$.
            For odd $n$ we get solutions $(d,k)=(2,1)$ and $(2,2)$.
         \item[(ii.b)] For $\ell\equiv 0\mod 4$ it is $\Lambda_R^{[\ell]}=
           \ell\car{\frac12 \alpha_1, \ldots \frac12 \alpha_{n-1}, \alpha_n}
           =2\Lambda_R^{(\ell)}$, $\ell_{[n]}=\ell$ and $m_{[n]}=n$. Thus the
           quotient diamonds and solutions are as in (ii).
        \end{enumerate}

      \item For $\g$ with root system $C_n$, $n\geq 3$, we have $\pi_1=\Z_2$
        for all $n$. The short simple roots are $\alpha_1,\ldots,\alpha_{n-1}$
        and the long simple root $\alpha_n$, hence $d_i=1$ for $1\leq i\leq
        n-1$ and $d_n=2$. The symmetrized Cartan matrix $\tilde C$ is given
        below. The fundamental dominant weights $\lambda_i$ are given as in
        \cite{Hum72}, Section 13.2, and $\lambda_n$ is the generator of the
        fundamental group $\Z_2$. The matrix $\id_W^R$ gives the coefficients
        of the fundamental dominant weights in the basis
        $\{\alpha_1,\ldots,\alpha_n\}$.

        \begin{equation*}
          \tilde C=
          \begin{pmatrix}
            2 & -1 & 0& 0& .& & & 0  \\
            -1&  2 &-1& 0& .& & & 0  \\
             .&   .& .& .& .&.&.&.&  \\
            0 &  0 & 0&  & .&-1&2&-2 \\
            0 &  0 & 0&  & .& 0&-2&4
          \end{pmatrix}
          \quad
          \id_W^R= 
          \begin{pmatrix}
            1 & 1 & . & . & 1 & 1 \\
            1 & 2 & . & . & 2 & 2 \\
            . & . & . & . & . & . \\
            1 & 2 & . & . & n-1 & n-1 \\
      \frac12 & 1 & . & . & \frac{n-1}2 & \frac n2
          \end{pmatrix}
        \end{equation*}

        The lattice diamonds, depending on $\ell$, are:
        \begin{enumerate}[(i)]
          \item For odd $\ell$ we have $A=B=\ell\Lambda_W$ and
            $C=D=\ell\Lambda_R$.
          \item For $\ell\equiv 2 \mod 4$ we have $A=
            \ell\car{\lambda_1,\ldots,\lambda_{n-1}, \frac 12\lambda_n}$ and
            $C=D=\ell\Lambda_W$. Since $(\lambda_n,\lambda_n)=n$,
            $B=\Cent^q(\Lambda_W)$ depends on $n$. For odd $n$ it equals
            $\ell\Lambda_W$ and for even $n$ it is equal to $A$.
          \item For $\ell\equiv 0 \mod 4$ we have $A=C=
            \ell\car{\lambda_1,\ldots,\lambda_{n-1},\frac 12\lambda_n}$ and
            $D=\ell\Lambda_W$. Here again, $B=\Cent^q(\Lambda_W)$ depends on
            $n$. For odd $n$ it equals $\ell\Lambda_W$ and for even $n$ it is
            equal to $A$.
        \end{enumerate}

        We calculate the quotient diamonds for kernel
        $\Lambda'=\Lambda_R^{[\ell]}$ and compare it with Lusztig's kernel
        $2\Lambda_R^{(\ell)}$. 
        We then determine the solutions of the corresponding diamond-equations
        according to Example \ref{expl:DiamondSolutionsCyclci}.
        \begin{enumerate}[(i)]
          \item For odd $\ell$ it is $\Lambda_R^{[\ell]}=\ell\Lambda_R \neq
            2\ell\Lambda_R = 2\Lambda_R^{(\ell)}$, $\ell_{[n]}=\ell$ and
            $m_{[n]}=2n$. In this case, the quotient diamond is given by
            $(\Z_2,\Z_2,\Z_2,\Z_1,\Z_1,1,\lambda_n)$. Thus we are in case
            (II) of Example \ref{expl:DiamondSolutionsCyclci}, hence the 2
            solutions are given by $(d,k)=(1,1)$ and $(2,1)$.
          \item[(ii)] For $\ell\equiv 2\mod 4$ it is $\Lambda_R^{[\ell]}=
            \ell\car{\alpha_1, \ldots ,\alpha_{n-1},\frac12  \alpha_n}
            \neq\ell\Lambda_R = 2\Lambda_R^{(\ell)}$, $\ell_{[n]}=\ell/2$ and
            $m_{[n]}=n$. The quotient diamond is given by either
            $(\Z_2,\Z_2,\Z_2,\Z_1,\Z_1,1,\lambda_n)$ for even $n$, or by
            $(\Z_2,\Z_2,\Z_1,\Z_1,\Z_1,-1,\lambda_n)$ for odd $n$. Thus
            we are in either in case (II) or in case (III) of Example
            \ref{expl:DiamondSolutionsCyclci}. In the first case (even $n$) we
            get solutions $(d,k)=(1,1)$ and $(2,1)$. For odd $n$ we get
            solutions $(d,k)=(1,1)$ and $(2,2)$.
         \item[(ii)] For $\ell\equiv 0\mod 4$ it is $\Lambda_R^{[\ell]}=
           \ell\car{\frac12 \alpha_1, \ldots \frac12 \alpha_{n-1}, \alpha_n}
           =2\Lambda_R^{(\ell)}$, $\ell_{[n]}=\ell/2$ and $m_{[n]}=n$. The
           quotient diamond is given by either
           $(\Z_2,\Z_2,\Z_2,\Z_2,\Z_1,1,0)$ for even $n$, or by
           $(\Z_2,\Z_2,\Z_1,\Z_2,\Z_1,-1,0)$ for odd $n$. Thus we
           are in either in case (I) or in case (III) of Example
           \ref{expl:DiamondSolutionsCyclci}. In the first case (even $n$) we
           get all 3 possible solutions $(d,k)=(1,1)$, $(2,1)$ and $(2,2)$. For
           odd $n$ we get solutions $(d,k)=(2,1)$ and $(2,2)$.  
        \end{enumerate}

      \item \label{D2k} For $\g$ with root system $D_n$, $n\geq 4$ even, we
        have $\pi_1=\Z_2\times \Z_2$ for all $n$. The simple roots are
        $\alpha_1,\ldots,\alpha_{n}$ and $d_i=1$ for $1\leq i\leq n$. The
        symmetrized Cartan matrix $\tilde C$ is given below. The fundamental
        dominant weights $\lambda_i$ are given as in \cite{Hum72}, Section
        13.2, and $\lambda_{n-1},\lambda_n$ are the generators of the
        fundamental group $\Z_2\times\Z_2$ and $\lambda_{n-1}+\lambda_n$ is
        the other element of order $2$. The matrix $\id_W^R$ gives the coefficients
        of the fundamental dominant weights in the basis
        $\{\alpha_1,\ldots,\alpha_n\}$, and since $d_i=1$ for all $i$, also
        the values $(\lambda_i,\lambda_j)$ for $1\leq i,j\leq n$.

        \begin{equation*}
          \tilde C=
          \begin{pmatrix}
            2 & -1 & 0& 0& .&  &  &   & 0  \\
            -1&  2 &-1& 0& .&  & .& . & .  \\
             .&  . & .& .& .& .& .& . & .  \\
             .&   .& .& .& .& 2&-1& 0 & 0  \\ 
             .&   .& .& .& .&-1& 2&-1 &-1  \\
            0 &  0 & 0&  & .& 0&-1& 2 &0 \\
            0 &  0 & 0&  & .& 0&-1& 0 &2
          \end{pmatrix}
          \quad
          \id_W^R= 
          \begin{pmatrix}
            1 & 1 & 1 & . & 1           & \frac12     & \frac12 \\
            1 & 2 & 2 & . & 2           & 1           & 1 \\
            1 & 2 & 3 & . & 3           & \frac32     & \frac32 \\
            . & . & . & . & .           & .           & . \\
            1 & 2 & 3 & . & n-2         & \frac{n-2}2 & \frac{n-2}2 \\
      \frac12 & 1 &\frac32 & . & \frac{n-2}2 & \frac n4    & \frac{n-2}4 \\
      \frac12 & 1 &\frac32 & . & \frac{n-2}2 & \frac{n-2}4 & \frac n4
          \end{pmatrix}
        \end{equation*}

        The lattice diamonds, depending on $\ell$, are:
        \begin{enumerate}[(i)]
          \item For odd $\ell$ we have $A=\ell\Lambda_W$ and
            $B=C=D=\ell\Lambda_R$.
          \item For even $\ell$ we have $A=C=\ell\Lambda_W$ and
            $B=D=\ell\Lambda_R$. 
        \end{enumerate}

        We calculate the quotient diamonds for kernel
        $\Lambda'=\Lambda_R^{[\ell]}$ and compare it with Lusztig's kernel
        $2\Lambda_R^{(\ell)}$. 
        We then determine the solutions of the corresponding diamond-equations
        by a case by case calculation.
        \begin{enumerate}[(i)]
          \item For odd $\ell$ it is $\Lambda_R^{[\ell]}=\ell\Lambda_R \neq
            2\ell\Lambda_R = 2\Lambda_R^{(\ell)}$. Thus, the quotient diamond
            is given by
            $(\Z_2\times\Z_2,\Z_2\times\Z_2,\Z_1,\Z_1,\Z_1,\varphi_1,\varphi_2)$
            with injections 
            \begin{align*}
              \varphi_1 &\colon \ell\car{\lambda_{n-1},\lambda_n} \to \pi_1^*, ~
              \ell\lambda_{n-1}\mapsto q^{\ell(\lambda_{n-1},-)},
              ~\ell\lambda_{n}\mapsto q^{\ell(\lambda_{n},-)}, \\
              \varphi_2 &\colon \ell\car{\lambda_{n-1},\lambda_n} \to \pi_1, ~
              \ell\lambda_{n-1}\mapsto \lambda_{n-1},~ \ell\lambda_{n}\mapsto
              \lambda_{n}.
            \end{align*}
            In the following, we will write $a\df\lambda_{n-1}$, $b\df\lambda_n$
            and $c\df\lambda_{n-1}+\lambda_n$ for the 3 elements of order 2 of
            $\pi_1$. Since $(\lambda_j,\lambda_j)=n/4$ for $j\in\{n-1,n\}$, and
            $(\lambda_i,\lambda_j)=(n-2)/4$ for $i\neq j$, $i,j\in\{n-1,n\}$
            we get\\

            \begin{center}
            \begin{tabular}{c||c|c|c|c||c}
              $\varphi_1$ & $0$ & $a$ & $b$ & $c$
              & $n$ \\\hline\hline
              \Multi{2}{$0$} &\Multi{2}{$1$} &\Multi{2}{$1$} &\Multi{2}{$1$}
              &\Multi{2}{$1$} & $n\equiv0\mod 4$\\\cline{6-6}
              &&&&& $n\equiv2\mod 4$\\\hline

              \Multi{2}{$a$} &\Multi{2}{$1$} &1 &-1 &-1
              &$n\equiv0\mod4$ \\\cline{3-6}
              &&-1 &1 &-1 &$n\equiv2\mod 4$\\\hline

              \Multi{2}{$b$} &\Multi{2}{$1$} &-1 &1 &-1
              &$n\equiv0\mod4$ \\\cline{3-6}
              &&1 &-1 &-1 &$n\equiv2\mod 4$\\\hline

              \Multi{2}{$c$} &\Multi{2}{$1$}
              &\Multi{2}{$-1$} &\Multi{2}{$-1$} &\Multi{2}{$1$}
              & $n\equiv0\mod4$ \\\cline{6-6}
              &&&&& $n\equiv2\mod 4$\\\hline
            \end{tabular}
            \end{center}
            ~\\

            Since it suffices to consider the diamond equations
            \eqref{diaeq03} and \eqref{diaeq04} by Lemma
            \ref{lem:onlyScalingEq}, we check which function
            \begin{equation*}
              g\colon G\times G\to \C, ~ (x,y) \mapsto \frac 1d 
              \, \omega(x,y)\delta_{(x\in H_1)}\delta_{(y\in H_2)}
            \end{equation*}
            with subgroups $H_i$ of $G=\Z_2\times\Z_2$ of order $d$ and a
            pairing $\omega$ as in Example \ref{2x2Solutions} is a solution to
            these equations.
            We get the following system of equations for $g$:
            \begin{align}\label{eqn}
              \begin{split}
              1 &= g(0,0) + g(a,0) + g(b,0) + g(c,0) \\
              1 &= g(0,0) + g(0,a) + g(0,b) + g(0,c) \\
              0 &= g(0,a) \pm g(a,a) \mp g(b,a) - g(c,a) \\
              0 &= g(a,0) \pm g(a,a) \mp g(a,b) - g(a,c) \\
              0 &= g(0,b) \mp g(a,b) \pm g(b,b) - g(c,b) \\
              0 &= g(b,0) \mp g(b,a) \pm g(b,b) - g(b,c) \\
              0 &= g(0,c) - g(a,c) - g(b,c) + g(c,c) \\
              0 &= g(c,0) - g(c,a) - g(c,b) + g(c,c) 
              \end{split}
            \end{align}
            where the $\pm,\mp$ possibilities depend on wether $\ell\equiv 0$
            or $2\mod 4$. It is easy to see that the trivial solution on
            $H_1=H_2=\Z_1$ is a solution. For $H_i\cong\Z_2$ the solution has
            one of the following two structures.  For symmetric solutions
            $H_1=H_2=\car{\lambda}$ we get $\omega(\lambda,\lambda)=-1$. If
            $H_1=\car{\lambda}\neq\car{\lambda'}=H_2$ we get
            $\omega(\lambda,\lambda')=1$. This give all possible 9 solutions
            with $H_i\cong\Z_2$. Finally, we check which functions on
            $G=\Z_2\times\Z_2$ are solutions to the diamond equations. 
            We get 4 symmetric solutions and 2 non-symmetric solutions, which
            are given by their values
            $(\omega(x,y))_{x,y\in\{\lambda_{n-1},\lambda_n\}}$ on generator pairs:
            \begin{equation*}
              \qquad
              \begin{pmatrix} 1 & 1 \\ 1 & 1 \end{pmatrix},~ 
              \begin{pmatrix}-1 &-1 \\-1 &-1 \end{pmatrix},~ 
              \begin{pmatrix} 1 & 1 \\ 1 &-1 \end{pmatrix},~ 
              \begin{pmatrix}-1 & 1 \\ 1 & 1 \end{pmatrix},~ 
              \begin{pmatrix}-1 & 1 \\-1 &-1 \end{pmatrix},~ 
              \begin{pmatrix}-1 &-1 \\ 1 &-1 \end{pmatrix}.
            \end{equation*}

          \item[(ii)] For even $\ell$ it is
            $\Lambda_R^{[\ell]}=\ell\Lambda_R =2\Lambda_R^{(\ell)}$. Thus the
            quotient diamond is given by
            $(\Z_2\times\Z_2,\Z_2\times\Z_2,\Z_1,\Z_2\times\Z_2,\Z_1,\varphi_1,0)$
            and the injection $\varphi_2$ is trivial. We get an analogue block
            of equations as \eqref{eqn}, but without non-zero ``shift''
            $\varphi_2(x)$, $x\in A$. We can add appropriate equations and
            get the $1=4g(0,0)$, hence only pairings of $H_1=H_2=\pi_1$ are
            solutions. It is now easy to check, that all 16 possible
            parings on $\pi_1\times \pi_1$ are solutions to the diamond
            equations.
        \end{enumerate}

      \item For $\g$ with root system $D_n$, $n\geq 5$ odd, we have
        $\pi_1=\Z_4$ for all $n$. The root and weight data are as for even $n$
        in \eqref{D2k}. The weight $\lambda_{n}$ is the generator of the
        fundamental group $\Z_2$.

        The lattice diamonds, depending on $\ell$, are:
        \begin{enumerate}[(i)]
          \item For odd $\ell$ we have $A=\ell\Lambda_W$ and
            $B=C=D=\ell\Lambda_R$.
          \item For $\ell\equiv 2 \mod 4$ we have $A=\ell\Lambda_W$, $C=
            \ell\car{\lambda_1,\ldots,\lambda_{n-2}, 2\lambda_{n-1},
            2\lambda_n}$ and $B=D=\ell\Lambda_R$. 
          \item For $\ell\equiv 0 \mod 4$ we have $A=C=\ell\Lambda_W$ and
            $B=D=\ell\Lambda_R$. 
        \end{enumerate}

        We calculate the quotient diamonds for kernel
        $\Lambda'=\Lambda_R^{[\ell]}$ and compare it with Lusztig's kernel
        $2\Lambda_R^{(\ell)}$. 
        We then determine the solutions of the corresponding diamond-equations
        according to Example \ref{expl:DiamondSolutionsCyclci}.
        \begin{enumerate}[(i)]
          \item For odd $\ell$ it is $\Lambda_R^{[\ell]}=\ell\Lambda_R \neq
            2\ell\Lambda_R = 2\Lambda_R^{(\ell)}$, $\ell_{[n]}=\ell$ and
            $m_{[n]}=n$. Thus, the quotient diamond is given by
            $(\Z_4,\Z_4,\Z_1,\Z_1,\Z_1,\xi_4, \lambda_n)$, hence we are in
            case (III) of Example \ref{expl:DiamondSolutionsCyclci}. We get
            solutions $(d,k)=(1,1)$, $(2,1)$, $(4,2)$ and $(4,4)$.
          \item For $\ell\equiv2\mod 4$ it is
            $\Lambda_R^{[\ell]}=\ell\Lambda_R= 2\Lambda_R^{(\ell)}$,
            $\ell_{[n]}=\ell$ and $m_{[n]}=n$. Thus, the quotient diamond is
            given by $(\Z_4,\Z_4,\Z_1,\Z_2,\Z_1,\xi_4, 2\lambda_n)$, hence we
            are in case (III) of Example \ref{expl:DiamondSolutionsCyclci}. We
            get all 4 solutions $(d,k)=(4,1)$, $(4,2)$, $(4,3)$ and $(4,4)$ on
            $H=\Z_4$. 
          \item For $\ell\equiv0\mod 4$ it is
            $\Lambda_R^{[\ell]}=\ell\Lambda_R= 2\Lambda_R^{(\ell)}$,
            $\ell_{[n]}=\ell$ and $m_{[n]}=n$. Thus, the quotient diamond is
            given by $(\Z_4,\Z_4,\Z_1,\Z_4,\Z_1,\xi_4, 0)$, hence we are in
            case (III) of Example \ref{expl:DiamondSolutionsCyclci}. We get
            the same 4 solutions as in (ii).
        \end{enumerate}

      \item For $\g$ with root system $E_6$, we have
        $\pi_1=\Z_3$. The simple roots are $\alpha_1,\ldots,\alpha_{6}$
        and $d_i=1$ for $1\leq i\leq 6$. The symmetrized Cartan matrix $\tilde
        C$ is given below. The fundamental dominant weights $\lambda_i$ are
        given as in \cite{Hum72}, Section 13.2, and $\lambda_6$ is the
        generator of the fundamental group $\Z_3$. The matrix $\id_W^R$ gives
        the coefficients of the fundamental dominant weights in the basis
        $\{\alpha_1,\ldots,\alpha_6\}$, and since $d_i=1$ for all $i$, also the
        values $(\lambda_i,\lambda_j)$ for $1\leq i,j\leq 6$.

        \begin{equation*}
          \tilde C=
          \begin{pmatrix}
            2& 0&-1& 0& 0& 0  \\
            0& 2& 0&-1& 0& 0  \\
           -1& 0& 2&-1& 0& 0  \\
            0&-1&-1& 2&-1& 0  \\ 
            0& 0& 0&-1& 2&-1  \\
            0& 0& 0& 0&-1& 2
          \end{pmatrix}
          \qquad
          \id_W^R= 
          \begin{pmatrix}
            \frac43 &1 &\frac53    &2 &\frac43    &\frac23 \\  
            1       &2 &2          &3 &2          &1       \\ 
            \frac53 &2 &\frac{10}3 &4 &\frac83    &\frac43 \\
            2       &3 &4          &6 &4          &2       \\ 
            \frac43 &2 &\frac83    &4 &\frac{10}3 &\frac53 \\
            \frac23 &1 &\frac43    &2 &\frac53    &\frac43
          \end{pmatrix}
        \end{equation*}

        The lattice diamonds, depending on $\ell$, are:
        \begin{enumerate}[(i)]
          \item For $3\nmid\ell$ we have $A=\ell\Lambda_W$ and
            $B=C=D=\ell\Lambda_R$.
          \item For $3\mid\ell$ we have $A=C=\ell\Lambda_W$ and
            $B=D=\ell\Lambda_R$.
        \end{enumerate}

        We calculate the quotient diamonds for kernel
        $\Lambda'=\Lambda_R^{[\ell]}$ and compare it with Lusztig's kernel
        $2\Lambda_R^{(\ell)}$. 
        We then determine the solutions of the corresponding diamond-equations
        according to Example \ref{expl:DiamondSolutionsCyclci}.
        \begin{enumerate}[(i)]
          \item[(i.a)] For $3\nmid \ell$ and $\ell$ odd it is
            $\Lambda_R^{[\ell]}=\ell\Lambda_R \neq 2\ell\Lambda_R =
            2\Lambda_R^{(\ell)}$, $\ell_{[n]}=\ell$ and $m_{[n]}=4$. Thus, the
            quotient diamond is given by
            $(\Z_3,\Z_3,\Z_1,\Z_1,\Z_1,\xi_3, \lambda_6)$, hence we are in
            case (III) of Example \ref{expl:DiamondSolutionsCyclci}. Since
            $\ell\equiv 2\mod 3$ we get solutions $(d,k)=(1,1)$, $(3,1)$ and
            $(3,3)$.
          \item[(i.b)] For $3\nmid \ell$ and $\ell$ even it is
            $\Lambda_R^{[\ell]}=\ell\Lambda_R = 2\Lambda_R^{(\ell)}$,
            $\ell_{[n]}=\ell$ and $m_{[n]}=4$. Thus, the quotient diamond is
            given by $(\Z_3,\Z_3,\Z_1,\Z_1,\Z_1,\xi_3, \lambda_6)$, and we are
            again in case (III) of Example \ref{expl:DiamondSolutionsCyclci}.
            Since $\ell\equiv 1\mod 3$ we get solutions $(d,k)=(1,1)$, $(3,2)$
            and $(3,3)$.
          \item[(ii.a)] For $3\mid \ell$ and $\ell$ odd it is
            $\Lambda_R^{[\ell]}=\ell\Lambda_R \neq 2\ell\Lambda_R =
            2\Lambda_R^{(\ell)}$, $\ell_{[n]}=\ell$ and $m_{[n]}=4$. Thus, the
            quotient diamond is given by
            $(\Z_3,\Z_3,\Z_1,\Z_3,\Z_1,\xi_3,0)$, hence we are in
            case (III) of Example \ref{expl:DiamondSolutionsCyclci}. We get
            all 3 solutions $(d,k)=(3,1)$, $(3,2)$ and $(3,3)$ on $\Z_3$.
          \item[(ii.b)] For $3\mid \ell$ and $\ell$ even it is
            $\Lambda_R^{[\ell]}=\ell\Lambda_R = 2\Lambda_R^{(\ell)}$,
            $\ell_{[n]}=\ell$ and $m_{[n]}=4$. Thus, the quotient diamond is
            given by $(\Z_3,\Z_3,\Z_1,\Z_3,\Z_1,\xi_3,0)$, and we the same
            solutions as in (ii.a).
        \end{enumerate}

      \item For $\g$ with root system $E_7$, we have
        $\pi_1=\Z_2$. The simple roots are $\alpha_1,\ldots,\alpha_{7}$
        and $d_i=1$ for $1\leq i\leq 7$. The symmetrized Cartan matrix $\tilde
        C$ is given below. The fundamental dominant weights $\lambda_i$ are
        given as in \cite{Hum72}, Section 13.2, and $\lambda_7$ is the
        generator of the fundamental group $\Z_2$. The matrix $\id_W^R$ gives
        the coefficients of the fundamental dominant weights in the basis
        $\{\alpha_1,\ldots,\alpha_7\}$, and since $d_i=1$ for all $i$, also the
        values $(\lambda_i,\lambda_j)$ for $1\leq i,j\leq 7$.

        \begin{equation*}
          \tilde C=
          \begin{pmatrix}
            2& 0&-1& 0& 0& 0 & 0 \\
            0& 2& 0&-1& 0& 0 & 0 \\
           -1& 0& 2&-1& 0& 0 & 0 \\
            0&-1&-1& 2&-1& 0 & 0 \\ 
            0& 0& 0&-1& 2&-1 & 0 \\
            0& 0& 0& 0&-1& 2 &-1 \\
            0& 0& 0& 0& 0&-1 & 2
          \end{pmatrix}
          \qquad
          \id_W^R= 
          \begin{pmatrix}
            2 &2       &3 &4 &3          &2 &1 \\
            2 &\frac72 &4 &6 &\frac92    &3 &\frac32 \\
            3 &4       &6 &8 &6          &4 &2 \\
            4 &6       &8 &12&9          &6 &3 \\
            3 &\frac92 &6 &9 &\frac{15}2 &5 &\frac52 \\
            2 &3       &4 &6 &5          &4 &2 \\
            1 &\frac32 &2 &3 &\frac52    &2 &\frac32 \\
          \end{pmatrix}
        \end{equation*}

        The lattice diamonds, depending on $\ell$, are:
        \begin{enumerate}[(i)]
          \item For odd $\ell$ we have $A=\ell\Lambda_W$ and
            $B=C=D=\ell\Lambda_R$.
          \item For even $\ell$ we have $A=C=\ell\Lambda_W$ and
            $B=D=\ell\Lambda_R$.
        \end{enumerate}

        We calculate the quotient diamonds for kernel
        $\Lambda'=\Lambda_R^{[\ell]}$ and compare it with Lusztig's kernel
        $2\Lambda_R^{(\ell)}$. 
        We then determine the solutions of the corresponding diamond-equations
        according to Example \ref{expl:DiamondSolutionsCyclci}.
        \begin{enumerate}[(i)]
          \item For $\ell$ odd it is $\Lambda_R^{[\ell]}=\ell\Lambda_R \neq
            2\ell\Lambda_R = 2\Lambda_R^{(\ell)}$, $\ell_{[n]}=\ell$ and
            $m_{[n]}=3$. Thus, the quotient diamond is given by
            $(\Z_2,\Z_2,\Z_1,\Z_1,\Z_1,\xi_2, \lambda_7)$ and we are in
            case (III) of Example \ref{expl:DiamondSolutionsCyclci}. We get
            solutions $(d,k)=(1,1)$ and $(2,2)$.
          \item For $\ell$ even it is $\Lambda_R^{[\ell]}=\ell\Lambda_R =
            2\Lambda_R^{(\ell)}$, $\ell_{[n]}=\ell$ and
            $m_{[n]}=3$. Thus, the quotient diamond is given by
            $(\Z_2,\Z_2,\Z_1,\Z_2,\Z_1,\xi_2,0)$ and we are again in
            case (III) of Example \ref{expl:DiamondSolutionsCyclci}. We get
            all 2 solutions $(d,k)=(2,1)$ and $(2,2)$ on $\Z_2$.
        \end{enumerate}

      \item For $\g$ with root system $E_8$, we have
        $\pi_1=\Z_1$. The simple roots are $\alpha_1,\ldots,\alpha_{8}$
        and $d_i=1$ for $1\leq i\leq 8$. The symmetrized Cartan matrix $\tilde
        C$ is given below. The fundamental dominant weights $\lambda_i$ are
        given as in \cite{Hum72}, Section 13.2, and are roots. The matrix
        $\id_W^R$ gives the coefficients of the fundamental dominant weights in
        the basis $\{\alpha_1,\ldots,\alpha_8\}$, and since $d_i=1$ for all
        $i$, also the values $(\lambda_i,\lambda_j)$ for $1\leq i,j\leq 8$.

        \begin{equation*}
          \tilde C=
          \begin{pmatrix}
            2& 0&-1& 0& 0& 0 & 0 & 0\\
            0& 2& 0&-1& 0& 0 & 0 & 0\\
           -1& 0& 2&-1& 0& 0 & 0 & 0\\
            0&-1&-1& 2&-1& 0 & 0 & 0\\ 
            0& 0& 0&-1& 2&-1 & 0 & 0\\
            0& 0& 0& 0&-1& 2 &-1 & 0\\
            0& 0& 0& 0& 0&-1 & 2 &-1\\
            0& 0& 0& 0& 0& 0 &-1 & 2
          \end{pmatrix}
          \quad
          \id_W^R= 
          \begin{pmatrix}
            4 & 5 & 7 &10 & 8 & 6 & 4 & 2 \\
            5 & 8 &10 &15 &12 & 9 & 6 & 3 \\
            7 &10 &14 &20 &16 &12 & 8 & 4 \\
           10 &15 &20 &30 &24 &18 &12 & 6 \\
            8 &12 &16 &24 &20 &15 &10 & 5 \\ 
            6 & 9 &12 &18 &15 &12 & 8 & 4 \\
            4 & 6 & 8 &12 &10 & 8 & 6 & 3 \\
            2 & 3 & 4 & 6 & 5 & 4 & 3 & 2
          \end{pmatrix}
        \end{equation*}

        The lattice diamonds, depending on $\ell$, are:
        \begin{enumerate}[(i)]
          \item For odd $\ell$ we have $A=B=C=D=\ell\Lambda_W=\ell\Lambda_R$. 
          \item For even $\ell$: same as in (i).
        \end{enumerate}

        We calculate the quotient diamonds for kernel
        $\Lambda'=\Lambda_R^{[\ell]}$ and compare it with Lusztig's kernel
        $2\Lambda_R^{(\ell)}$. 
        We then determine the solutions of the corresponding diamond-equations
        according to Example \ref{expl:DiamondSolutionsCyclci}.
        \begin{enumerate}[(i)]
          \item For $\ell$ odd it is $\Lambda_R^{[\ell]}=\ell\Lambda_R \neq
            2\ell\Lambda_R = 2\Lambda_R^{(\ell)}$, $\ell_{[n]}=\ell$ and
            $m_{[n]}=2$. Thus, the quotient diamond is given by
            $(\Z_1,\Z_1,\Z_1,\Z_1,\Z_1,1,0)$ and we are in case (I) of Example
            \ref{expl:DiamondSolutionsCyclci}. We get the only solution
            $(d,k)=(1,1)$.
          \item For $\ell$ even it is $\Lambda_R^{[\ell]}=\ell\Lambda_R =
            2\Lambda_R^{(\ell)}$, $\ell_{[n]}=\ell$ and
            $m_{[n]}=2$. We get the same diamond and solution as in (i).
        \end{enumerate}

      \item For $\g$ with root system $F_4$, we have
        $\pi_1=\Z_1$. The simple roots $\alpha_1,\alpha_{2}$ are long,
        $\alpha_3,\alpha_4$ are short, hence
        $d_1=d_2=2$ and $d_3=d_4=1$. The symmetrized Cartan matrix $\tilde
        C$ is given below. The fundamental dominant weights $\lambda_i$ are
        given as in \cite{Hum72}, Section 13.2, and are roots. The matrix
        $\id_W^R$ gives the coefficients of the fundamental dominant weights in
        the basis $\{\alpha_1,\ldots,\alpha_4\}$.

        \begin{equation*}
          \tilde C=
          \begin{pmatrix}
            4&-2& 0& 0\\
           -2& 4&-2& 0\\
            0&-2& 2&-1\\
            0& 0&-1& 2
          \end{pmatrix}
          \qquad
          \id_W^R= 
          \begin{pmatrix}
            4 & 6 & 4 & 2  \\
            6 &12 & 8 & 4  \\
            4 & 8 & 6 & 3  \\
            2 & 4 & 3 & 1 
          \end{pmatrix}
        \end{equation*}

        The lattice diamonds, depending on $\ell$, are:
        \begin{enumerate}[(i)]
          \item For odd $\ell$, we have $A=B=C=D=\ell\Lambda_W=\ell\Lambda_R$.
          \item For even $\ell$, we have $A=B=C=D=\ell\car{\frac12
            \lambda_1,\frac12\lambda_2,\lambda_3,\lambda_4}$.
        \end{enumerate}

        We calculate the quotient diamonds for kernel
        $\Lambda'=\Lambda_R^{[\ell]}$ and compare it with Lusztig's kernel
        $2\Lambda_R^{(\ell)}$. 
        We then determine the solutions of the corresponding diamond-equations
        according to Example \ref{expl:DiamondSolutionsCyclci}.
        \begin{enumerate}[(i)]
          \item For $\ell$ odd it is $\Lambda_R^{[\ell]}=\ell\Lambda_R \neq
            2\ell\Lambda_R = 2\Lambda_R^{(\ell)}$, $\ell_{[n]}=\ell$ and
            $m_{[n]}=1$. Thus, the quotient diamond is given by
            $(\Z_1,\Z_1,\Z_1,\Z_1,\Z_1,1,0)$ and we are in case (I) of Example
            \ref{expl:DiamondSolutionsCyclci}. We get the only solution
            $(d,k)=(1,1)$.
          \item For $\ell\equiv2\mod4$ it is
            $\Lambda_R^{[\ell]}=\ell\car{\frac12\alpha_1,
            \frac12\alpha_2,\alpha_3, \alpha_4} \neq
            \ell\Lambda_R = 2\Lambda_R^{(\ell)}$, $\ell_{[n]}=\ell$ and
            $m_{[n]}=1$. We get the same diamond and solution as in (i).
          \item For $\ell\equiv0\mod4$ it is
            $\Lambda_R^{[\ell]}=\ell\car{\frac12\alpha_1,
            \frac12\alpha_2,\alpha_3, \alpha_4}=2\Lambda_R^{(\ell)}$,
            $\ell_{[n]}=\ell$ and $m_{[n]}=1$. We get the same diamond and
            solution as in (i).
        \end{enumerate}

      \item For $\g$ with root system $G_2$, we have
        $\pi_1=\Z_1$. The simple root $\alpha_1$ is short and $\alpha_2$ is long, hence
        $d_1=1$ and $d_2=3$. The symmetrized Cartan matrix $\tilde
        C$ is given below. The fundamental dominant weights $\lambda_i$ are
        given as in \cite{Hum72}, Section 13.2, and are roots. The matrix
        $\id_W^R$ gives the coefficients of the fundamental dominant weights in
        the basis $\{\alpha_1,\ldots,\alpha_2\}$.

        \begin{equation*}
          \tilde C=
          \begin{pmatrix}
            2&-3\\
           -3& 6
          \end{pmatrix}
          \qquad
          \id_W^R= 
          \begin{pmatrix}
            2 & 3  \\
            1 & 2 
          \end{pmatrix}
        \end{equation*}

        The lattice diamonds, depending on $\ell$, are:
        \begin{enumerate}[(i)]
          \item For $3\nmid\ell$, we have
            $A=B=C=D=\ell\Lambda_W=\ell\Lambda_R$.
          \item For $3\mid\ell$, we have $A=B=C=D=\ell\car{\lambda_1,\frac13
            \lambda_2}$.
        \end{enumerate}

        We calculate the quotient diamonds for kernel
        $\Lambda'=\Lambda_R^{[\ell]}$ and compare it with Lusztig's kernel
        $2\Lambda_R^{(\ell)}$. 
        We then determine the solutions of the corresponding diamond-equations
        according to Example \ref{expl:DiamondSolutionsCyclci}.
        \begin{enumerate}[(i)]
          \item[(i.a)] For $3\nmid\ell$ and $\ell$ odd it is
            $\Lambda_R^{[\ell]}=\ell\Lambda_R \neq 2\ell\Lambda_R =
            2\Lambda_R^{(\ell)}$, $\ell_{[n]}=\ell$ and $m_{[n]}=6$. Thus, the
            quotient diamond is given by $(\Z_1,\Z_1,\Z_1,\Z_1,\Z_1,1,0)$ and
            we are in case (I) of Example \ref{expl:DiamondSolutionsCyclci}.
            We get the only solution $(d,k)=(1,1)$.
          \item[(i.b)] For $3\nmid\ell$ and $\ell$ even it is
            $\Lambda_R^{[\ell]}=\ell\Lambda_R = 2\Lambda_R^{(\ell)}$,
            $\ell_{[n]}=\ell$ and $m_{[n]}=6$. We get the same diamond and
            solution as in (i.a).
          \item[(ii.a)] For $3\mid\ell$ and $\ell$ odd it is
            $\Lambda_R^{[\ell]}=\ell\car{\alpha_1,
            \frac13\alpha_2} \neq 2\ell\car{\alpha_1,
            \frac13\alpha_2} = 2\Lambda_R^{(\ell)}$,
            $\ell_{[n]}=\ell/3$ and $m_{[n]}=2$. We get the same diamond and
            solution as in (i.a).
          \item[(ii.b)] For $3\mid\ell$ and $\ell$ even it is
            $\Lambda_R^{[\ell]}=\ell\car{\alpha_1,
            \frac13\alpha_2} = 2\Lambda_R^{(\ell)}$,
            $\ell_{[n]}=\ell/3$ and $m_{[n]}=2$. We get the same diamond and
            solution as in (i.a).
        \end{enumerate}
    \end{enumerate}

%

\end{document}